\documentclass[twocolumn]{autart}
\usepackage{graphicx}
\usepackage{subcaption}

\usepackage{algorithm}
\usepackage{acronym}
\usepackage{algorithmic}
\usepackage{natbib}       									% need for \citet*
\usepackage{amsmath}                        % need for subequations

\usepackage{graphicx}                       % need for figures
\usepackage{enumitem}                       % need for subfigures

\usepackage{amssymb}                        % gives you \mathbb{} font
\usepackage{mathtools}                      % need for `show only references'
\mathtoolsset{showonlyrefs=true}            % only equations which are labeled AND referenced will numbered

\MakeRobust{\eqref}

\newtheorem{theorem}{Theorem}
\newtheorem{lemma}[theorem]{ {Lemma}}       	

\newtheorem{corollary}[theorem]{ {Corollary}}
\newtheorem{definition}[theorem]{ {Definition}}

\newtheorem{remark}[theorem]{ {Remark}}
\newtheorem{assumption}[theorem]{ {Assumption}}

\usepackage[usenames, dvipsnames, svgnames, table]{xcolor}

\renewcommand\limsup{\operatorname*{lim \, sup}}
\renewcommand\liminf{\operatorname*{lim \, inf}}

\usepackage{tikz} 
\usetikzlibrary{matrix}
\usepackage{hyperref}
\hypersetup{colorlinks=true, allcolors=black}

	\hypersetup{
		colorlinks   = true,
		citecolor    = RoyalBlue,
		linkcolor    = RubineRed,
		urlcolor     = Turquoise
		}
\newcommand{\behcet}{Beh\c{c}et}
\newcommand{\acikmese}{A\c{c}\i kme\c{s}e}

\newcommand{\MR}{\mathbb{R}}

\newcommand{\Rfunc}[3]{{#1} \colon \mathbb{R}^{#2} \to \mathbb{R}^{#3}}
\newcommand{\Tfunc}[2]{{#1} \colon [0, t_f] \to \mathbb{R}^{#2}}
\newcommand{\intset}[2]{\left[#1,#2\right]_{\mathbb{N}}}
\newcommand{\blue}[1]{{#1}}
\usepackage{etoolbox} 
\apptocmd{\normalsize}{\setlength{\abovedisplayskip}{5pt plus 2pt minus 0pt}}{}{}
\apptocmd{\normalsize}{\setlength{\belowdisplayskip}{1pt plus 1pt minus 0pt}}{}{}
\apptocmd{\normalsize}{\setlength{\abovedisplayshortskip}{0pt plus 0pt}}{}{}
\apptocmd{\normalsize}{\setlength{\belowdisplayshortskip}{3pt plus 2pt minus 0pt}}{}{}

\acrodef{LCvx}{Lossless Convexification}

\begin{document}
\begin{frontmatter}
\title{Revisiting Lossless Convexification: Theoretical Guarantees for Discrete-time Optimal Control Problems }
\thanks[footnoteinfo]{Corresponding author Dayou Luo. }

\author[Dayou]{Dayou Luo },    % Add the 
\author[AA]{Kazuya Echigo},               % e-mail address 
~and~
\author[AA]{\text{{\behcet~\acikmese}}} % (ead) as shown

\address[Dayou]{Department of Applied Mathematics, University of Washington, Seattle, WA, 98195}  % Please supply        
\address[AA]{William E. Boeing Department of Aeronautics and Astronautics, University of Washington, Seattle, WA, 98195}                            
       % full addresses

\maketitle
\begin{abstract}

\ac{LCvx} is a modeling approach that transforms a class of nonconvex optimal control problems, where nonconvexity primarily arises from control constraints, into convex problems through convex relaxations. These convex problems can be solved using polynomial-time numerical methods after discretization, which converts the original infinite-dimensional problem into a finite-dimensional one. However, existing \ac{LCvx} theory is limited to continuous-time optimal control problems, as the equivalence between the relaxed convex problem and the original nonconvex problem holds only in continuous-time. This paper extends \ac{LCvx} theory to discrete-time optimal control problems by classifying them into normal and long-horizon cases. For normal cases, after an arbitrarily small perturbation to the system dynamics (recursive equality constraints), applying the existing \ac{LCvx} method to discrete-time problems results in optimal controls that meet the original nonconvex constraints at all but no more than \(n_x - 1\) temporal grid points, where \(n_x\) is the state dimension. For long-horizon cases, the existing \ac{LCvx} method fails, but we resolve this issue by integrating it with a bisection search, leveraging the continuity of the value function from the relaxed convex problem to achieve similar results as in normal cases. This paper strengthens the theoretical foundation of \ac{LCvx}, extending the applicability of \ac{LCvx} theory to discrete-time optimal control problems.

\end{abstract}
\end{frontmatter}

\section{Introduction}
% \red{todo: Acknowledge}

% \ac{LCvx} is a modeling method developed for a class of nonconvex continuous-time optimal control problems, where the nonconvexity is due to the control constraints. \ac{LCvx} converts a nonconvex optimal control problem into an equivalent convex problem via novel relaxations.

% \ac{LCvx} is a method designed to transform a class of nonconvex continuous-time optimal control problems, where the nonconvexity is primarily due to control constraints, into equivalent convex problems through novel relaxations.
%  Using Pontryagin's maximum principle, the solution to the relaxed convex problem satisfies the constraints of the original nonconvex problem~\citep{acikmese2005powered,berkovitz1975}. Introducing \ac{LCvx} simplifies the computation for the global minimizer of the nonconvex optimal control problem: it transforms nonconvex problems, which are difficult to solve numerically, into convex problems, whose discretizations can be efficiently solved using off-the-shelf solvers, such as ECOS~\citep{ECOS2013} or Mosek~\citep{mosek}. The paper aims to extend the use of the \ac{LCvx} method from continuous-time optimal control problems to discrete-time optimal control problems.

\acresetall
% \ac{LCvx} is a method designed to transform a class of nonconvex continuous-time optimal control problems, where the nonconvexity is primarily due to control constraints, into equivalent convex problems through novel relaxations. Using Pontryagin's maximum principle, it can be shown that the solution to the relaxed convex problem still satisfies the original problem's constraints~\citep{acikmese2005powered,berkovitz1975}. This relaxation simplifies the computation of the global minimizer by converting challenging nonconvex problems into convex ones, whose discretizations can be efficiently solved using off-the-shelf solvers, such as ECOS~\citep{ECOS2013} or Mosek~\citep{mosek}. Building on these strengths, this paper aims to extend the applicability of the \ac{LCvx} family from continuous-time to discrete-time optimal control problems.

\ac{LCvx} is a technique designed to efficiently solve a class of nonconvex continuous-time optimal control problems, where nonconvexity primarily arises from control constraints. \blue{The central idea of LCvx is to introduce slack variables to relax these nonconvex constraints, thereby converting the nonconvex feasible set into a convex one.} Importantly, under mild conditions, Pontryagin's maximum principle guarantees that the solution of this relaxed convex formulation remains feasible and optimal for the original nonconvex problem. \blue{As this relaxation introduces no approximation error or loss in solution quality, the technique is termed ``lossless"~\citep{acikmese2005powered}.} Once reformulated, the continuous time problem can be discretized into a finite-dimensional convex optimization problem and efficiently solved using standard polynomial-time interior solvers \citet{nemirovski2004interior}.

\ac{LCvx} research has been active for the past decade, and it has proven useful in numerous applications ~\citep{malyuta2022convex}. The initial application of the \ac{LCvx} method targeted Mars pinpoint landings~\citep{acikmese2005powered, acikmese2007convex}, where the nonzero lower bound on the magnitude of the control input vector led to nonconvex control constraints. Later, the \ac{LCvx} family has been generalized to handle a wide range of continuous-time optimal control problems, including but not limited to those with generalized nonconvex input constraints~\citep{accikmecse2011lossless, yang2023convex}, input pointing constraints~\citep{carson2011lossless}, integer constraints~\citep{Harris2021, woodford2021geometric,Kazu2023, echigo2024expected, malyuta2020lossless}, active fault diagnosis\citep{guo2024active}, as well as problems with additional convex affine and quadratic state constraints~\citep{harris2013lossless, harris2014lossless}.
 Moreover, the efficacy of the \ac{LCvx} method has also been tested in real-world experiments~\citep{acikmese2012g, acikmese2013flight} and is now being used as a baseline for Moon landing applications\citep{fritz2022post, shaffer2024implementation}. 
% For a comprehensive review of \ac{LCvx}’s history, applications, and research directions, readers can refer to~\citep{malyuta2022convex}.

% However, there is a gap between \ac{LCvx} theory and its applications: 
% Although \ac{LCvx} performs well in real-world applications, its existing theoretical framework does not offer rigorous guarantees for practical use in real-world scenarios. The existing \ac{LCvx} theory demonstrates an equivalence between a continuous-time nonconvex optimal control problem and its convex relaxation. However, the latter is still a continuous-time problem and cannot be solved using numerical solvers, which only handle finite-dimensional problems. Therefore, in practice, we convert both problems into finite-dimensional discrete-time problems for numerical solutions~\citep{accikmecse2011lossless}. The main issue with this approach is that the existing \ac{LCvx} theory offers no guarantees for this discretization step—specifically, whether the optimal solution of the discrete-time relaxed convex optimal control problem still satisfies the constraints of the discrete-time nonconvex problem.

Despite the success of \ac{LCvx} in real-world applications, \blue{its theoretical results do not fully support its practical use}. While \ac{LCvx} theory establishes an equivalence between a continuous-time nonconvex optimal control problem and its convex relaxation, both remain continuous-time problems and cannot be directly solved by numerical solvers, which handle only finite-dimensional problems. To address this, discretization is usually applied to transform both problems into finite-dimensional, discrete-time versions~\citep{accikmecse2011lossless}. However, this discretization step introduces additional theoretical challenges, as continuous-time \ac{LCvx} theory does not guarantee that the optimal solution of the discretized convex relaxation satisfies the constraints of the discretized nonconvex problem, \blue{and thus it does not guarantee the equivalence between these two discrete problems}.

In fact, the optimal solution of the discrete-time convex relaxation can violate the constraints of the original discrete-time nonconvex problem, \blue{a phenomenon previously observed in LCvx research \citep{acikmese2007convex, malyuta2020lossless}.} In this paper, we provide numerical examples showing that the solution to the relaxed convex problem may fail to satisfy the original nonconvex constraints at one, multiple, or even all temporal grid points. \blue{These examples highlight potential challenges in directly extending lossless convexification from continuous-time to discrete-time settings.}

\blue{Motivated by these observations, the primary goal of this paper is to develop a rigorous theoretical framework for directly applying lossless convexification to discrete-time nonconvex optimal control problems. Specifically, we provide formal reasoning to explain why constraint violations occur in discrete-time settings and establish theoretical bounds on the number of temporal grid points at which these violations can appear.
}

% In fact, the optimal solution of the discrete-time convex relaxed problem can violate the constraints of the original discrete-time nonconvex problem. In this paper, we provide numerical examples demonstrating that the optimal solution of the relaxed convex problem may fail to satisfy the original nonconvex constraints at one, multiple, or all temporal grid points. \blue{These examples illustrate the potential limitations of applying lossless convexification to discrete-time nonconvex optimal control problems.}

% \blue{To address this issue, the primary goal of this work is to develop a rigorous theoretical framework for directly applying lossless convexification to discrete-time nonconvex optimal control problems. Specifically, we provide formal reasoning to explain the loss of equivalence between the convexified and original discrete-time problems and establish theoretical bounds on the resulting constraint violations.}

Instead of considering a general class of nonconvex control constraints, we focus on a simple problem where the lower bound of the nonzero norm on the control input magnitude is the sole source of nonconvexity. Here, the input magnitude is defined by a convex function of the input variable. We chose this problem because it can demonstrate the essence of \ac{LCvx} method without introducing excessive mathematical complexity. The application of the \ac{LCvx} method to a broader class of nonconvex problems will be explored in future studies.

In our theoretical framework for the discrete-time nonconvex problem, we classify such problems into two categories: normal cases and long-horizon cases. Long-horizon cases arise when the control horizon is sufficiently long to admit an optimal solution for the convexified problem, but the input magnitude of this solution consistently violates the nonconvex input constraints throughout the trajectory. All other problems are classified as normal cases. \blue{The existing \ac{LCvx} method has demonstrated effective performance for normal cases~\citep{acikmese2005powered, acikmese2007convex}, but it fails for long-horizon cases because a key assumption for LCvx theory—the strict transversality assumption (\citep[Condition 2]{malyuta2022convex}, \citep[Condition 2]{accikmecse2011lossless})—is no longer valid.}

% For normal cases, our method is identical to the existing \ac{LCvx} method, as the existing approach has demonstrated effective performance for such cases~\citep{acikmese2005powered, acikmese2007convex}. However, in long-horizon cases, the existing approach fails to provide feasible solutions for the nonconvex problem. To address this, we introduce a bisection search strategy specifically designed to handle long-horizon cases.

% For normal cases, we provide theoretical guarantees for applying the conventional \ac{LCvx} to discrete-time optimal control problems in the normal case setting. Specifically, we show that the optimal solution to the convexified problem violates the original nonconvex constraints at no more than \( n_x - 1 \) temporal grid points after an arbitrarily small perturbation to the dynamics, where \( n_x \) is the state dimension. We also provide numerical examples to demonstrate the necessity of the perturbation and the tightness of the upper bound \( n_x - 1 \).  Our proof closely relates to the discretized Pontryagin's Maximum Principle~\citep{Sethi2021Optimal, canon1970theory}, as both approaches are based on the Lagrangian function for finite-dimensional optimal control problems.

For the normal case, we provide a theoretical guarantee for applying the existing \ac{LCvx} method to discrete-time optimal control problems by proving a bound on the number of points where the optimal solution to the convexified problem can violate the original nonconvex constraints. Specifically, we show that, after an arbitrarily small perturbation to the dynamics, the optimal solution to the convexified problem violates nonconvex constraints at no more than \( n_x - 1 \) temporal grid points, where \( n_x \) is the state dimension. Furthermore, we show that the perturbed problem maintains all necessary assumptions from the relaxed convex problem and we establish bounds on the impact of the perturbation. Numerical examples are also provided to demonstrate the necessity of the perturbation and the tightness of the upper bound. Our proof technique is closely related to the discretized Pontryagin's Maximum Principle~\citep{Sethi2021Optimal, canon1970theory}, as both approaches are based on the Lagrangian function for finite-dimensional optimal control problems. However, unlike the discretized Pontryagin's Maximum Principle, our method does not involve the use of the Hamiltonian.

%from previous research~\citep{acikmese2005powered, acikmese2007convex}

In contrast, the existing \ac{LCvx} method does not work for long-horizon cases \blue{because the strict transversality assumptions (\citep[Condition 2]{malyuta2022convex}, \citep[Condition 2]{accikmecse2011lossless}) required by existing \ac{LCvx} methods no longer hold for such cases}. To address this, we propose a \blue{modification of} the \ac{LCvx} method by combining it with a bisection search. Specifically, we divide the control horizon into two phases. In the first phase, we use an arbitrary feasible control input to minimize cost. The final state of the first phase is then used as the initial state for the second phase, where we apply the existing \ac{LCvx} method to the resulting nonconvex optimal control problem. We expect the control horizon for the second phase to be short enough to satisfy the transversality assumption and fall into the normal case category, while still being long enough to maintain minimum input cost. To achieve this balance, we develop a bisection search to determine the optimal transition time between phases, leveraging the continuity of the value function of the resulting convex optimal control problem. By employing a bisection search, the second-phase control problem falls into the normal case category, enabling us to directly apply the theoretical results of normal cases to long-horizon problems.

\blue{We note that other methods also address cases where the strict transversality assumptions no longer hold. For instance, \citep{kunhippurayil2021lossless} applies a bisection search in continuous-time settings. Our work differs from \citep{kunhippurayil2021lossless} by focusing on discrete-time problems, incorporating terminal costs into the objective rather than fixed terminal states, and performing the bisection search on a different, more practical discrete-time objective. Another long-horizon approach, presented in \citep{yang2024exact}, introduces auxiliary state variables to indirectly address the issue. In contrast, we explicitly analyze why LCvx fails in long-horizon scenarios and propose a direct solution requiring fewer assumptions. This solution transforms the long-horizon problem into a normal one, allowing us to directly apply theoretical results developed for normal cases and facilitating extensions to other discrete-time LCvx methods, such as \citep{malyuta2020lossless}.
}

\textbf{Statement of Contributions:} In this paper, we develop a new \ac{LCvx} \blue{theoretical framework} that directly applied to discrete-time nonconvex optimal control problems, bridging the gap between \ac{LCvx} theory and its real-world applications. Specifically, our contributions are as follows.

First, we develop a new theoretical framework for discrete-time optimal control problems, which differs from previous \ac{LCvx} research and other asymptotic theories for discretizations of continuous-time optimal control~\citep{ober2011discrete,lee2022convexifying}. By focusing on discrete-time settings, we can provide a tight upper bound for the number of temporal grid points at which the optimal control violates the nonconvex constraints after applying the \ac{LCvx} method.

Second, our theoretical framework employs linear algebra and finite-dimensional convex analysis, thus avoiding the need for infinite-dimensional mathematical tools such as Pontryagin's maximum principle. This approach not only simplifies the theoretical development but also offers a fresh perspective on the mechanism of \ac{LCvx}.

Third, our theoretical framework eliminates the requirement for the strict transversality assumption found in earlier \ac{LCvx} theories (\citep[Condition 2]{malyuta2022convex} and~\citep[Condition 2]{accikmecse2011lossless}), an assumption that is often difficult to verify before computation. This removal makes our LCvx theory more applicable to real-world applications.

\textbf{Road map:} In Section \ref{sec: background}, we review the background of \ac{LCvx} research and define the optimization problem under consideration. Next, in Section \ref{sec: Mechanism of LCvx}, we provide an intuitive explanation of the mechanism of \ac{LCvx} using convex analysis techniques and define the mathematical criteria for normal cases and long-horizon cases. In Section \ref{sec: normal cases}, we discuss the validity of \ac{LCvx} for normal cases, define perturbations to the discrete dynamics, and provide theoretical guarantees for the \ac{LCvx} method on the perturbed nonconvex optimal control problem. In Section \ref{sec: long-horizon}, we address the long-horizon case and propose a bisection search algorithm for these cases. Finally, in Section \ref{sec:numerical}, we present numerical experiments to demonstrate counterexamples where the \ac{LCvx} method is invalid and showcase the perturbation and bisection search algorithms.

\subsection{Notation}
In the remainder of this discussion, we adopt the following notation. The set of integers $\{a, a+1, \ldots, b-1, b\}$ is denoted by $[a,b]_{\mathbb{N}}$. The set of real numbers is denoted by $\mathbb{R}$, the set of real $n \times m$ matrices by $\mathbb{R}^{n \times m}$, and the set of real $n$-dimensional vectors by $\mathbb{R}^n$. Vector inequalities are defined elementwise. For matrices $A_1, A_2, \ldots, A_n$, where each matrix $A_i$ has the same number of rows, we denote the horizontal concatenation of these matrices by $(A_1, A_2, \ldots, A_n).$ This operation results in a matrix formed by placing matrices $A_1$ through $A_n$ side by side, from left to right. The Euclidean norm for vectors and matrices is denoted by $\|\cdot \|$, and the absolute value by $|\cdot|$. For a matrix $A$, $\operatorname{exp}(A)$ denotes the matrix exponential of $A$, with the definition available in~\citep{chen1984linear} and $\operatorname{rank}(A)$ denotes the rank of matrix $A$.  For a convex set $\mathbb{D}$, the normal cone at the point $z$ is denoted by $N_{\mathbb{D}}(z)$, with the definition found in~\citep[Chapter 2.1]{borwein2006convex}. For a function $f(x,y)$ convex in $(x,y)$, $\partial_x f(x,y)$ and $\partial_x f(x,y)$ represent the subgradient of $f$ with respect to the variables $x$ and $y$ separately, as defined in~\citep{boyd2004convex}. For a differentiable function $f(x,y)$, $\nabla_x f(x,y)$ and $\nabla_y f(x,y)$ denote the partial derivative of $f$ with respect to $x$ and $y$ separately. Finally, $\operatorname{Prob}(E)$ denotes the probability that certain events $E$ occur.

\section{Background}
\label{sec: background}
In this section, we discuss the background of the \ac{LCvx} method. Consider the optimization problem mentioned in~\citep[Problem 6]{malyuta2022convex}, which includes no state constraints other than boundary conditions. We further simplify this problem to a time-invariant linear system without external disturbances. We believe that this basic form captures the essential ideas of the \ac{LCvx} method without adding mathematical complexity. This leads to the following optimal control problem:
\newline
\textbf{Problem 1}:
\label{equ:unslacked_continuous}
\begin{equation}
\begin{aligned}
 \min _{t_f, u, x}\;& m\left(x\left(t_f\right)\right)+\int_0^{t_f} l_c(g(u(t))) \, dt, \\
 \operatorname{s.t.} \quad & \dot{x}(t)=A_c x(t)+B_c u(t),\\
& \rho_{\min} \leq g(u(t))\leq \rho_{\max}, \text{ for } t \in [0, t_f]\\
& x(0)=x_{\text{init}}, \quad G(x(t_f)) = 0.
\end{aligned}
\end{equation}
Here, $\Rfunc{m}{n_x}{},\ \Rfunc{l_c}{}{}$, and $\Rfunc{g}{n_u}{}$ are assumed to be convex $C^1$ smooth functions. $\Rfunc{G}{n_x}{n_G}$ is an affine map. The state trajectory $\Tfunc{x}{n_x}$ and the control input trajectory $\Tfunc{u}{n_u}$ satisfy a continuous ordinary differential equation constraint. The condition $G(x(t_f)) = 0$ represents the boundary condition for the state trajectory $x$. $A_c \in \MR^{n_x \times n_x}$ and $B_c \in \MR^{n_x \times n_u}$ are constant matrices. $t_f \in \MR$ is a free flight time and $x_{\text{init}} \in \MR^{n_x}$ is a fixed initial state. $\rho_{min} \in \MR$ and $\rho_{max} \in \MR$ are fixed constants. The nonconvexity of Problem~\hyperref[equ:unslacked_continuous]{1} arises solely from potentially nonconvex constraints $\rho_{\min} \leq g(u(t))$, where $g$ characterize the magnitude of the input variable. A classic realization of
 $g$ is to use the norm function. Such constraints arise from planetary soft landing applications, where the net thrust of rocket engines must exceed a positive lower bound to prevent shutdown~\citep{acikmese2005powered}.

The \ac{LCvx} technique introduced in~\citep{accikmecse2011lossless} considers the following convex relaxation of the Problem~\hyperref[equ:unslacked_continuous]{1}:
\newline
\textbf{Problem 2}:
\label{equ:slackened}
\begin{equation}
\begin{array}{ll}
\min _{t_f, \sigma, u, x} & m\left(x\left(t_f\right)\right)+\int_0^{t_f} l_c(\sigma(t)) d t \\
\text { s.t. } & \dot{x}(t)=A_c x(t)+B_c u(t) \\
&\blue{ \rho_{\min } \leq \sigma(t) \leq \rho_{\min },\ g(u(t)) \leq \sigma(t)} \\
& x(0)=x_{\text {init }}, \quad G\left(x\left(t_f\right)\right)=0
\end{array}
\end{equation}
where $\Tfunc{\sigma}{}$ is a slack variable. \blue{The introduction of the slack variable \(\sigma\) makes the feasible set convex. Under mild conditions, the dual variables associated with the constraint \(g(u(t))\le\sigma(t)\) are positive almost everywhere, enforcing the equality \(g(u(t))=\sigma(t)\) for almost every \(t\), as formally established in Theorem~\citep[Theorem~2]{accikmecse2011lossless}. Consequently, the solution to the convexified problem satisfies the original nonconvex constraint \(\rho_{\min}\le g(u(t))\le\rho_{\max}\) for almost every \(t\).}

% Since the infinite-dimensional optimization problem Problem~\hyperref[equ:slackened]{2} is difficult to solve numerically, discretization of Problem~\hyperref[equ:slackened]{2} is used to form the control input~\citep{acikmese2005powered, acikmese2007convex, accikmecse2011lossless}. In the following, we consider applying zero-order hold discretization to both Problem~\hyperref[equ:unslacked_continuous]{1} and Problem~\hyperref[equ:slackened]{2}. Additionally, having a free flight time $t_f$ may undermine the convexity of the resulting optimization problem. A practical way to address this difficulty is by choosing a finite number of candidates for $t_f$, solving the discretized \ac{LCvx} convex optimization problem defined by these $t_f$s, and selecting the one with the smallest objective value~\citep{acikmese2007convex}.

Now, we consider the discretized, finite-dimensional,  versions of Problem~\hyperref[equ:unslacked_continuous]{1} and Problem~\hyperref[equ:slackened]{2}. We apply a temporal discretization with zero-order hold for both infinite-dimensional problems to solve them numerically~\citep{acikmese2005powered, acikmese2007convex, accikmecse2011lossless}. Meanwhile, having a free final time $t_f$ can undermine the convexity of the resulting optimization problem. To address this, a commonly used method is to apply a line search for $t_f$. By solving the discretized relaxed problem for each $t_f$, one can find a $t_f$ with the smallest objective value~\citep{acikmese2007convex}. Hence, in the rest of this paper, we only consider the cases where $t_f$ is fixed.

With the following solution  variables:
\begin{equation}
\begin{aligned}
x = & \{x_1, x_2, \cdots, x_{N+1}\}, \  x_i \in \mathbb{R}^{n_x} \text{ for } i \in \intset 1 {N+1}\\
u = & \{u_1, u_2, \cdots, u_N\},\  u_i \  \in \mathbb{R}^{n_u} \text{ for } i \in \intset 1 N\\
\sigma = & \{\sigma_1, \sigma_2, \cdots, \sigma_{N}\}, \  \sigma_i \in \mathbb{R} \text{ for } i \in \intset 1 N.
\end{aligned}
\end{equation}
The  discretized version  of Problem~\hyperref[equ:unslacked_continuous]{1} is:
\newline
\textbf{Problem 3}:
\label{equ:original_LCvx}
\begin{equation}
\begin{array}{ll}
\underset{x,u}{\min}& m\left(x_{N+1}\right) + \sum_{i = 1}^N l_i\left(g(u_i)\right) \\
 \operatorname{s.t.} \quad & x_{i+1} = A_d x_{i} + B_d u_i,\quad i \in \intset 1 N \\
& \rho_{\min} \leq g(u_i) \leq \rho_{\max},\quad i \in \intset 1 N\\
& x_1 = x_{\text{init}}, \quad G\left(x_{N+1}\right) = 0.
\end{array}  
\end{equation}
The sum of $\Rfunc{l_i}{}{}$ is an approximation for the integral term in the objective function of Problem~\hyperref[equ:unslacked_continuous]{1}. Matrices $A_d\in \MR^{n_x \times n_x}$ and $B_d\in \MR^{n_x \times n_u}$ for this discretized problem are given as~\citep{chen1984linear}:
\begin{equation}
\label{equ: dicrete dyanmics matrix}
A_d=\exp \left(A_c \frac{t_f}{N}\right), \quad B_d=\int_0^{t_f / N} \exp \left(A_c \tau\right) d \tau B_c
\end{equation}
Since we mainly work with the discretized problem, we simplify the notation by using $A$ for $A_d$ and $B$ for $B_d$. Condition $G(x_{N+1}) = 0$ is denoted as the boundary condition, $x_{i+1} = A x_{i} + B u_i$ as the dynamics constraints, and $m(x_{N+1}) +\Sigma_i l_i(\sigma_i)$ as the object function. 

The discretized version of  Problem~\hyperref[equ:slackened]{2} is:
\newline
\textbf{Problem 4}:
\label{equ:zero_order_prototype}
\begin{equation}
\begin{array}{ll}
\underset{x,u,\sigma}{\min} & m\left(x_{N+1}\right) + \sum_{i = 1} ^Nl_i\left(\sigma_i\right) \\
 \operatorname{s.t.} \quad & x_{i+1} = A x_{i} + B u_i,\quad i \in \intset 1 N \\
& \rho_{\min} \leq \sigma_i\leq \rho_{\max},\ g(u_i) \leq \sigma_i,\ i \in \intset 1 N \\
& x_1 = x_{\text{init}}, \quad G\left(x_{N+1}\right) = 0. 
\end{array}    
\end{equation}
Problem~\hyperref[equ:zero_order_prototype]{4} can be understood as the result of applying the \ac{LCvx} method to Problem~\hyperref[equ:original_LCvx]{3}. Since Problem~\hyperref[equ:zero_order_prototype]{4} is a relaxation of Problem~\hyperref[equ:original_LCvx]{3}, its optimal value serves as a lower bound for that of Problem~\hyperref[equ:original_LCvx]{3}.

Motivated by previous \ac{LCvx} research on continuous problems, a natural question arises: does the optimal solution to Problem~\hyperref[equ:zero_order_prototype]{4} satisfy the constraints of Problem~\hyperref[equ:original_LCvx]{3}? However, this relationship does not hold in general. The numerical counterexamples in Section \ref{sec:numerical} show that the solution to Problem~\hyperref[equ:zero_order_prototype]{4} often does not satisfy the constraints of Problem~\hyperref[equ:original_LCvx]{3} at certain temporal grid points. Therefore, an optimal solution to Problem~\hyperref[equ:zero_order_prototype]{4} is not guaranteed to solve Problem~\hyperref[equ:original_LCvx]{3}. As a result, we shifted our focus to a revised question. At how many temporal grid points can the solution to Problem~\hyperref[equ:zero_order_prototype]{4} potentially violate the constraints of Problem~\hyperref[equ:original_LCvx]{3}, thereby invalidating \ac{LCvx}?
Here, we define the validity of \ac{LCvx} at grid point $i$ as:
\begin{definition}
The \ac{LCvx} method is valid at a grid point \( i \) if the control input $u_i$ satisfies \(\rho_{\min}\leq g(u_i) \leq \rho_{\text{max}}. \) 
\end{definition} 

We will later prove that, with reasonable assumptions on Problem~\hyperref[equ:zero_order_prototype]{4}  and with arbitrarily small random perturbations to the discrete dynamics, \ac{LCvx} is invalid at most at \(n_x - 1\) temporal grid points, where \(n_x\) denotes the dimension of the variable \(x\).

Before the theoretical discussion, we provide some definitions and assumptions for the problem~\hyperref[equ:zero_order_prototype]{4} required for our discussion.

\begin{assumption}
\label{assumption: uicompactness}
$l_i(\cdot)$ is $C^1$ smooth and monotonically increasing on $[\rho_{min}, \infty)$. The function $m$ is also $C^1$ smooth. The set $\{u \mid g(u_i) \leq \rho_{min} \text{ for all } i\}$ is nonempty, and the set $\{u \mid g(u_i) \leq \rho_{max} \text{ for all } i\}$ is compact.  
\end{assumption}

%Assumption \ref{assumption: monotone in sigma original} keeps the possibility of having $\rho_{min} = \sigma_i$ for \eqref{equ:zero_order_prototype}. If the first part of the Assumption \ref{assumption: monotone in sigma original} is not satisfied, the lower bound of $\sigma_i$ would become the minimizer of $l_i(\cdot)$ on $[\rho_{min}, \infty)$, which is possible to be different from $\rho_{min}$. As a result, the constraint $\rho_{\text{min}} \leq \sigma_i$ may never be satisfied as an equality. In such a case, it would be appropriate to replace $\rho_{min}$ by the minimizer of $l_i(\cdot)$ on $[\rho_{min}, \infty)$. For the second half Assumption \ref{assumption: monotone in sigma original}, if it is not valid, the \ac{LCvx} at $v_i$ will be trivially satisfied. 

% Note that if $\{u| g(u_i) \leq \rho_{min}  \text{ for all } i\}$ is empty, \ac{LCvx} is trivially valid. The last part of Assumption \ref{assumption: uicompactness} indicates that in Problem~\hyperref[equ:original_LCvx]{3}, $\{u| g(u_i) \leq \rho_{min}  \text{ for all } i\} \subset \{u| g_0(u_i) \leq \rho_{max} \text{ for all } i\}$. A similar assumption is required in~\citep{accikmecse2011lossless}. For further details, see the discussion under~\citep[Equation (4)]{accikmecse2011lossless}.

We also require that the dynamical system in Problem~\hyperref[equ:zero_order_prototype]{4} is controllable. Note that the controllability of the discretized system is closely related to that of the continuous system~\citep[Theorem 6.9]{chen1984linear}.  

\begin{assumption}[Controllability]
\label{assumption:controllability}
For matrices \( A \) and \( B \) in Problem~\hyperref[equ:zero_order_prototype]{4},
\begin{equation}
\operatorname{rank}\left(\begin{array}{cccc}
B & AB & \cdots & A^{n_x-1}B
\end{array}\right) = n_x,
\end{equation}
i.e., the matrix pair \( \{A,B\} \) is controllable. 
\end{assumption}

We now introduce some notation for further assumptions. Since the constraints on $(u, \sigma)$ in Problem~\hyperref[equ:zero_order_prototype]{4} are identical at each temporal grid point, we define  
\begin{equation}
 V :=\{(\bar u, \bar \sigma) \in \mathbb{R}^{n_u}\times \mathbb{R} \,|\, \rho_{\min} \leq \bar \sigma \leq  \rho_{\max}, \, g(\bar u) \leq \bar \sigma\},
\end{equation}
so that the constraint in Problem~\hyperref[equ:zero_order_prototype]{4} can be expressed as $(u_i, \sigma_i) \in V$ for all $i$. Since $g$ is convex, the set $V$ is a convex set.

Moreover, we define $V(\cdot)$ as $
V(\bar \sigma) := \{\bar u \in \mathbb{R}^{n_u} \,| g(\bar u) \leq \bar \sigma \},$
i.e. $V(\bar \sigma)$ is a slice of $V$ at a specific value of $\bar \sigma \in \mathbb R$. The definition of \( V(\cdot) \) establishes a sufficient condition for the validity of \ac{LCvx} for certain \( u_i \). The following lemma follows immediately.

%%%%%%%%%%%%%%%%%%%%%%%%%%%
\begin{lemma}
\label{lemma: LCvx v(sigma)}
Under Assumption \ref{assumption: uicompactness}, if a control \( \bar u \) lies on the boundary of \( V(\bar \sigma) \) , \( \bar u \) satisfies the constraints in Problem~\hyperref[equ:original_LCvx]{3}, i.e. $ \rho_{min} \leq g(\bar u)= \bar \sigma \leq \rho_{max}$.
\end{lemma}

%%%%%%%%%%%%%%%%%%%%%%%%%%%
To define other assumptions, we need to further simplify the notation for Problem~\hyperref[equ:zero_order_prototype]{4}. Let \( z := (x, u, \sigma)\in \MR^{n_z} \) with $n_z = (N+1) n_x + N n_u + N$. We denote the object function as \( \Rfunc{\Phi}{n_z}{}\) and encapsulate the dynamics and boundary conditions into an affine constraint \( H(z) = 0 \), where $\Rfunc{H}{n_z}{n_H}$ and $n_H$ represents the total number of equality constraints. The remaining inequality constraints are written within the convex set \( \tilde{V} = \{(x,u,\sigma) \in \MR^{n_z}\,|\, (u_i, \sigma_i) \in V \text{  for all } i \} \). Thus, Problem~\hyperref[equ:zero_order_prototype]{4} can be represented as:
\newline
\textbf{Problem 5}:
\label{equ:simplified_zero_order_hold}
\begin{equation}
\begin{aligned}
 \min_z \quad& \Phi\left(z\right) \\
 \operatorname{s.t.} \quad & H\left(z\right) = 0, \quad z \in \tilde{V}.
\end{aligned}
\end{equation}
Since \( H \) is an affine function, the Jacobian of \( H \), denoted as \(\nabla H(z)\), is a constant matrix. Therefore, by eliminating the equations in \( H \) if necessary, \(\nabla H(z)\) becomes a matrix of full rank. Since the dynamics equalities are designed to be full rank, any elimination of redundant constraints occurs only at the boundary condition and does not affect the dynamics equalities. After this elimination, we make the following assumption:

\begin{assumption}
\label{assumption: dynamics full rank}
For Problem~\hyperref[equ:simplified_zero_order_hold]{5}, \(\nabla H(z)\) has full row rank.    
\end{assumption}

For the convex Problem~\hyperref[equ:simplified_zero_order_hold]{5}, we further require the Slater condition:

\begin{assumption}[Slater's condition]
\label{assumption: slater}
For Problem~\hyperref[equ:simplified_zero_order_hold]{5}, there exists a vector $z$ in the interior of $\tilde V$ such that $H(z) = 0$ 
\end{assumption}

Assumption \ref{assumption: slater} is relatively modest. When $g$ is the Euclidean norm function, Assumption \ref{assumption: slater} is equivalent to having a control sequence \(\|u_i\| < \rho_{max}\) such that the trajectory generated by this \( u_i \) satisfies \( G(x_{N+1}) = 0 \). This holds when the optimal control problem is not overly constrained.

In the following discussion of this paper, Assumption \ref{assumption: uicompactness}, \ref{assumption:controllability},  \ref{assumption: dynamics full rank},  and \ref{assumption: slater} are valid for Problem~\hyperref[equ:zero_order_prototype]{4} without further mentioning.

We conclude the background section with a continuity theorem, which is a major tool in our subsequent theoretical discussion. We begin by introducing a lemma which is a direct consequence of ~\citep[Lemma 3.3.3, Theorem 3.3.4]{clarke2008nonsmooth}:

\begin{lemma}
\label{thm:clarkeexistance}
Consider an optimization problem 
\newline
\textbf{Problem 6}:
\label{equ: general pertube}
\begin{equation}
\begin{aligned}
 \min_y \quad & M\left(y, p\right) \\
 \operatorname{s.t.} \quad & F\left(y, p\right) = 0, \quad y \in \Omega,
\end{aligned}
\end{equation}
where $y \in \MR^{n_y}$, $p \in \MR^{n_p}$. $\Rfunc{M}{n_y + n_p}{}$ and $\Rfunc{F}{n_y + n_p}{n_F}$ are continuous with respect to $(y,p)$, and $\nabla_y F$ exists and is continuous with respect to $(y,p)$.  
Consider a solution $\left(y^*, p_0\right)$ to the equality/constraint system
\begin{equation*}
F(y^*, p_0) = 0, \quad y^* \in \Omega,
\end{equation*}
such that  $\eta = 0 $ is the unique solution to the following equation 
\begin{equation}
0 \in \nabla_y F\left(y^*, p_0\right)^\top \eta + N_\Omega\left(y^*\right).
\end{equation}
Then, there exist $\delta, \xi > 0$ such that for any $p \in \MR^{n_p}$ with $\|p-p_0\| < \xi$, there is a point $y \in \Omega$ satisfying $F(y, p) = 0$ and 
$$
\|y - y^*\| \leq \frac{\|F(y^*, p)\|}{\delta}.
$$
\end{lemma}

We denote the optimum value to Problem~\hyperref[equ: general pertube]{6} as $M^*(p)$. Our goal here is to show the continuity of $M^*$.

\begin{theorem}
\label{thm: vt continuity}
Assume that the function $F$ and an optimal solution $(y^*, p_0)$ to Problem~\hyperref[equ: general pertube]{6} satisfy all the assumptions in Lemma \ref{thm:clarkeexistance}, $\Omega$ is compact, and $M$ is a locally Lipschitz function with respect to $(y, p)$. Then, The optimal value function $\Rfunc{M^*}{n_p}{}$ to Problem~\hyperref[equ: general pertube]{6} is continuous at $p_0$.
\end{theorem}
\vspace{-5mm} 
\begin{proof}
See Appendix \ref{sec: proofs for background}.
\qed \end{proof}
\vspace{-2mm} 

%%%%%%%%%%%%%%%%%%%%%%%%%%%%%%%%%%
%%\section{Mechanism of LCvx}
%%%%%%%%%%%%%%%%%%%%%%%%%%%%%%%%%%%
\section{Mechanism of LCvx}
\label{sec: Mechanism of LCvx}
In this section, we provide a sufficient condition for the validity of \ac{LCvx} and introduce the long-horizon cases and the normal cases for Problem~\hyperref[equ:zero_order_prototype]{4}.

To start, we define the Lagrangian of Problem~\hyperref[equ:zero_order_prototype]{4} as follows:
\begin{equation}
\begin{aligned}
&L(x, u, \sigma ; \eta, \mu_1, \mu_2):=m\left(x_{N+1}\right)+\sum_{i=1}^{N} I_{V}\left(u_i, \sigma_i\right) \\
& +\sum_{i=1}^N l_i\left( \sigma_i\right) +\sum_{i=1}^{N} \eta_i^\top(-x_{i+1}+A x_{i}+B u_i) \\
& +\mu_1^\top\left(x_1-x_{init}\right)+\mu_2^\top G\left(x_{N+1}\right),
\end{aligned}
\end{equation}
where $I_V$ is the indicator function of set $V$:
\begin{equation}
 I_{V}(\bar u, \bar \sigma)=\left\{\begin{array}{cc}
0 & \text{if } (\bar u, \bar \sigma) \in V, \\
+\infty & \text{otherwise},
\end{array}\right.
\end{equation}
and the dual variable $\eta_i \in \MR^{n_x}$ for $i = 1, 2, \cdots, N, $ $\eta = [\eta_1^\top, \eta_2^\top, \cdots, \eta_N^\top]^\top \in \mathbb{R}^{n_x N}, $ $\mu_i \in \MR^{n_x}, $ and $ \mu_2 \in \MR^{n_G}.$ 
%%%%%%%%%%%%%%%%%%%%%%%%%%%%%%%%%%%%%%%%%%%%%%%%
%%%%%%%%%%%%%%%%%%%%%%%%%%%%%%%%%%%%%%%%%%%%%%%%
\begin{theorem}
\label{thm: lagrangian}
Suppose Assumption \ref{assumption: slater} holds, and let $(x^*, u^*, \sigma^*)$ be a solution of Problem~\hyperref[equ:zero_order_prototype]{4}. Then, there exist dual variables $(\eta^*, \mu_1^*, \mu_2^*)$ such that
\begin{equation}
\label{equ: min lagrangian}
\left(x^*, u^*, \sigma^* \right)=\operatorname{arg} \min_{x,u, \sigma} L\left(x, u, \sigma; \eta^*, \mu_1^*, \mu_2^*\right).
\end{equation}
\end{theorem}
\begin{proof}
This follows directly from~\citep[Theorem 9.4, 9.8 and 9.14]{clarke2013functional}.
\qed \end{proof}
%%%%%%%%%%%%%%%%%%%%%%%%%%%%%%%%%%%%%%%%%%
From equation \eqref{equ: min lagrangian}, we derive:
\begin{align}
  &  \sigma^* =\operatorname{arg} \min_\sigma L\left(x^*, u^*, \sigma; \eta^*, \mu_1^*, \mu_2^*\right),  \\
  & u^*=\operatorname{arg} \min_u L\left(x^*,u, \sigma^*; \eta^*, \mu_1^*, \mu_2^*\right), \label{minimization on u} \\
  & x^* =\operatorname{arg} \min_x L\left(x, u^*, \sigma^*; \eta^*, \mu_1^*, \mu_2^*\right). \label{minimization on x}
\end{align}
Since in the Lagrangian $L$, $u_i$ only appears in terms $I_V(u_i, \sigma_i)$ and $\eta_i^\top B u_i$, we can write:
\begin{equation}
    \label{equ:ui}
\begin{aligned}
u_i^* & =\arg \min _{u_i} I_V\left(u_i, \sigma_i^*\right)+\eta_i^{* \top} B u_i \\
& =\arg \min _{u_i} I_{V\left(\sigma_i^*\right)}\left(u_i\right)+\eta_i^{* \top} B u_i .
\end{aligned}
\end{equation}
Equation \eqref{equ:ui} characterizes the relationship between $u_i^*$ and $\sigma_i^*$. With \eqref{equ:ui} in hand, we discuss the sufficient condition for \ac{LCvx} to be valid.
%%%%%%%%%%%%%%%%%%%%%%%%%%%%%%%%%%%%%%%%%
\begin{theorem}[Sufficient Condition for LCvx]
\label{lemma:LCvx lambda B}
Under Assumptions \ref{assumption: uicompactness} and \ref{assumption: slater}, the condition of $-B^{\top}\eta_i ^{*}$ being nonzero is sufficient for $u_i$ to lie on the boundary of $V(\sigma_i^*)$, thus validating \ac{LCvx} at the grid point $i$ by Lemma \ref{lemma: LCvx v(sigma)}. Furthermore, $\eta_i^*$ is given by $\eta_i^* = A^{\top (N-i)} \eta_N^*$.
\end{theorem}
%%%%%%%%%%%%%%%%%%%%%%%%%%%
\begin{proof}
We first address the sufficient condition for \( u_i^* \) to lie on the boundary of $V(\sigma_i^*)$. Given that the right-hand side function of equation \eqref{equ:ui} is convex in \( u_i \) and \( u_i^* \) is a minimizer of this convex function, it follows that\citep[Proposition 2.1.2]{borwein2006convex} 
\[
-B^{\top}\eta_i ^{*} \in N_{V(\sigma_i^*)}\left(u_i^*\right),
\]
where \( N_{V(\sigma_i^*)}\left(u_i\right) \) denotes the normal cone of the set \( V(\sigma_i^*) \) at point \( u_i^* \). The normal cone has nonzero elements only at the boundary points of \( V(\sigma_i^*) \)~\citep[Corollary 6.44]{bauschke2011convex}. Hence, \( -B^{\top}\eta_i ^{*} \) being nonzero implies that \( u_i \) lies on the boundary of \( V(\sigma_i^*) \). In such a case, \ac{LCvx} is valid at node $i$ according to Lemma \ref{lemma: LCvx v(sigma)}.

For the evolution of \( \eta_i \), considering equation \eqref{minimization on x}, we take the partial derivative of \( L \) with respect to \( x_i \) for all \( i \), yielding:
\begin{subequations}
\begin{align}
x_{N+1}: & \quad \nabla m\left(x_{N+1}\right) + \mu_2^\top \nabla G\left(x_{N+1}\right) = \eta^*_N, \label{eq:xN}\\
x_i: & \quad \eta_i^{*\top} A - \eta_{i-1}^{*\top} = 0, \label{eq:xi}\\
x_1: & \quad \mu_1^\top + \eta_1^{*\top} A = 0. \label{eq:x0}
\end{align}
\end{subequations}
Thus, iteratively applying \eqref{eq:xi} leads to  \( \eta^*_i = A^{\top(N-i)} \eta^*_N \). 
\qed \end{proof}

%%%Add a discussion here. 
\begin{remark}
  Theorem \ref{lemma:LCvx lambda B} offers insight into why discrete \ac{LCvx} is typically valid in applications. We start by designing a slack variable $\sigma$ such that for a fixed $\sigma$, \ac{LCvx} is valid as long as the control $u$ is on the boundary of the feasible region defined by $\sigma$. Then, if a dual variable \(\eta_i^*\) is not in the left null space of matrix \(B\), the corresponding control \(u\) lies on the boundary of the feasible region, thus validating \ac{LCvx}.

\end{remark}

However, Theorem \ref{lemma:LCvx lambda B} does not guarantee the validity of \ac{LCvx}, as \( -B^{\top}\eta_i^* = 0 \) can occur for some \( i \). An obvious scenario for \( \eta_i^* = 0 \) is where \( \eta_N^* = 0 \). The following lemma characterizes this scenario.

\begin{lemma}
\label{thm: eta_n is zero}
Under Assumptions \ref{assumption: uicompactness} and \ref{assumption: slater}, if \( \eta_N^* = 0 \), then all \( \sigma_i^* = \rho_{\min} \), and \( x_{N+1} \) must be an optimal solution to the following optimization problem:
\newline
\textbf{Problem 7}:
\label{equ: boundaryonly}
\begin{equation}
\begin{aligned}
 \min_z \quad &m\left(x_{N+1}\right) \\
 \operatorname{s.t.} \quad &  G\left(x_{N+1}\right)=0.
\end{aligned}
\end{equation}
\end{lemma}

\begin{proof}
From Theorem \ref{lemma:LCvx lambda B}, \( \eta_N^* = 0 \) implies \( \eta_i^* = 0 \) for all $i$. By Theorem \ref{thm: lagrangian}, we have
\[
(\sigma^*, u^*)=\operatorname{arg} \min_{\sigma, u }\sum_{i=1}^N l_i\left( \sigma_i\right)+\sum_{i=1}^N I_V\left(u_i, \sigma_i\right).
\]
Given Assumption \ref{assumption: uicompactness}, we have that \( \sigma_i^* = \rho_{\min} \) and \( u_i \in V(\rho_{\min}) \).

We now prove the result on $x_{N+1}$. Equation \eqref{eq:xN} ensures that
\[
\nabla m\left(x_{N+1}\right)+\mu_2^\top \nabla G\left(x_{N+1}\right)=\eta^*_N = 0,
\]
indicating that \( x_{N+1} \) is a KKT point of the convex optimization Problem~\hyperref[equ: boundaryonly]{7} and thus a minimizer~\citep{borwein2006convex}. 
\qed \end{proof}

Theorem \ref{lemma:LCvx lambda B} provides a sufficient condition for \( \eta_N^* \neq 0 \), leading to a natural division of Problem~\hyperref[equ:zero_order_prototype]{4} into two scenarios:
\begin{definition}
Problem~\hyperref[equ:zero_order_prototype]{4} is defined as the long-horizon case if \( \sigma_i^* = \rho_{\min} \) and \( x_{N+1} \) solves Problem~\hyperref[equ: boundaryonly]{7}. If these conditions are not satisfied, Problem~\hyperref[equ:zero_order_prototype]{4} is classified as the normal case.
\end{definition}

The term long-horizon is used because \( \eta_N^* = 0 \) typically occurs when the control horizon is large. For a problem in the normal case, \( \eta_N^* \neq 0 \) according to Lemma \ref{thm: eta_n is zero}. In previous \ac{LCvx} research, \( \eta_N^* \neq 0 \) is usually ensured by assumptions, for example~\citep[Condition 2]{malyuta2022convex} and~\citep[Condition 2]{accikmecse2011lossless}. By identifying the long-horizon case, we can omit such types of assumptions.

In the next section, we discuss the theoretical guarantee for normal cases. The long-horizon cases will be left for the subsequent section. 
%%%%%%%%%%%%%%%%%%%%%%%%%%%%%%%%%%%%%%%%%%%%
%%%%%%%%%%%%%%%%%%%%%%%%%%%%%%%%%%%%%%%%%
\section{LCvx on Normal cases}
\label{sec: normal cases}

In this section, we concentrate on problems in the normal case and thus we assume that

\begin{assumption}
\label{assumption: eta_N nonzero}
Problem~\hyperref[equ:zero_order_prototype]{4} is normal, i.e., either $\sigma_i > \rho_{\min}$ for some $i$ or $x_{N+1}$ is not the minimizer for Problem~\hyperref[equ: boundaryonly]{7}.
\end{assumption}
As discussed previously, Assumption \ref{assumption: eta_N nonzero} implies that $\eta_N^* \neq 0$. A direct consequence of this assumption is:

\begin{corollary}
\label{thm: total controllable B}
Given Assumptions \ref{assumption: uicompactness}, \ref{assumption: slater} and \ref{assumption: eta_N nonzero}, if the matrix $A$ and $B$ are invertible, then \ac{LCvx} is valid at all temporal grid points.
\end{corollary}

\begin{proof}
When $\eta_N^* \neq 0$, since $\eta^*_i=A^{\top(N-i)} \eta^*_N$, a full rank matrix $A$ ensures that $\eta_i^*$ is non-zero for all $i$. The invertibility of $B$ implies that $\eta_i^{*\top} B$ is non-zero for every $i$, satisfying the conditions of Theorem \ref{lemma:LCvx lambda B} for \ac{LCvx} to hold.
\qed \end{proof}

However, Corollary \ref{thm: total controllable B} might not always be applicable because it requires matrix $B$ to be a square matrix, which means that the control and state dimensions must be equal; In practice, the state dimension often significantly exceeds the control dimension.

When the state dimension is larger than the control dimension, the validity of \ac{LCvx} at all grid points is not guaranteed. Numerical examples in Section \ref{sec:numerical} show the invalidity of \ac{LCvx} at multiple grid points for a non-square matrix $B$, even when the \ac{LCvx} problem is normal. Therefore, we aim to establish a weaker statement: the original nonconvex constraints are violated at most at \(n_x - 1\) temporal grid points, where \(n_x\) is the state dimension. However, this statement is still invalid as counterexamples exist (see Section \ref{sec:Artificial normal example }). Thus, we consider an even weaker statement: after adding a random perturbation to the dynamics matrix $A$ in Problem~\hyperref[equ:zero_order_prototype]{4}, the nonconvex constraints are violated at no more than $n_x - 1$ temporal grid points.

\subsection{Pertubation}
We define an arbitrarily small perturbation in the context of matrix $A$'s Jordan form, where $A$ is defined in the dynamics constraints of Problem~\hyperref[equ:original_LCvx]{3} and Problem~\hyperref[equ:zero_order_prototype]{4}:
\begin{definition}
\label{def: perturbation}
Let matrix $A$'s Jordan form be
$$
A = Q^{-1} J Q.
$$  
Let $\left\{\lambda_1, \ldots, \lambda_d\right\}$ be the set of all distinct eigenvalues of $A$. Define $\tilde A(q)$ as a perturbation of $A$ with respect to a vector $q=(q_1, q_2, \cdots, q_d) \in \mathbb{R}^d$ if:
\begin{itemize}
    \item $\tilde A(q) = Q^{-1} \tilde J Q$, where $\tilde{J}$ is a Jordan matrix that is identical to $J$ except for its diagonal elements.
    \item $\tilde A(q)$ has $d$ eigenvalues $(\tilde \lambda_1, \cdots, \tilde\lambda_d)$ such that $$(\tilde \lambda_1, \cdots, \tilde\lambda_d) = (\lambda_1, \lambda_2, \cdots, \lambda_d) + (q_1, q_2, \cdots, q_d).$$
    \item The index of Jordan blocks in $\tilde{J}$ corresponding to $\tilde{\lambda}_i$ remains the same as in $J$ corresponding to $\lambda_i$.
\end{itemize}
\end{definition}
In essence, perturbing a matrix involves altering its eigenvalues. We intentionally preserve the Jordan form as it may convey important structural information about the dynamical system.

We now perturb Problem~\hyperref[equ:zero_order_prototype]{4} by replacing the matrix \( A \) in the dynamic constraints with \(\tilde{A}(q)\), where \(\tilde{A}(q)\) represents a perturbation of \( A \) with respect to the vector \( q \in \mathbb{R}^d \). Consequently, the optimization problem can be reformulated as follows:
\newline
\textbf{Problem 8}:
\label{equ:perturbed_zero_order_prototype}
\begin{equation}
\begin{array}{ll}
\underset{x,u,\sigma}{\min}& m\left(x_{N+1}\right)+\sum_{i = 1}^N  l_i\left( \sigma_i\right) \\
 \operatorname{s.t.} \quad & x_{i+1}=\tilde{A}(q) x_{i} + B u_i, \quad i\in \intset 1 N \\
& \rho_{\min } \leq \sigma_i \leq \rho_{\max },\ g(u_i) \leq \sigma_i,\ i\in \intset 1 N \\
& x_1 = x_{init}, \quad G\left(x_{N+1}\right)=0. 
\end{array}    
\end{equation}
If $q = 0,$ we restore Problem~\hyperref[equ:zero_order_prototype]{4}.

In the remainder of this paper, \(\mathbb{D} \subset \mathbb{R}^n\) denotes an \(n\)-dimensional unit cube, defined as \(\mathbb{D} = \{x \in \mathbb{R}^n \mid -1 \leq x_i \leq 1, \forall i = 1, \ldots, n\}\), where \(x_i\) is the \(i\)-th component of \(x\) in \(\mathbb{R}^n\). For \(\epsilon > 0\), we define the scaled cube \(\epsilon\mathbb{D}\) as \(\{x \in \mathbb{R}^n \mid -\epsilon \leq x_i \leq \epsilon, \forall i = 1, \ldots, n\}\). The following theorem ensures that Assumptions  \ref{assumption:controllability}, \ref{assumption: dynamics full rank}, \ref{assumption: slater}, and \ref{assumption: eta_N nonzero} remain valid for the perturbed problem~\hyperref[equ:perturbed_zero_order_prototype]{8} under sufficiently small perturbations \(q \in \epsilon \mathbb{D}\).

\begin{lemma}
\label{lemma: perturbation_all_stability}
Suppose that Assumptions \ref{assumption: uicompactness}, \ref{assumption:controllability},  \ref{assumption: dynamics full rank}, \ref{assumption: slater} and  \ref{assumption: eta_N nonzero} hold for Problem~\hyperref[equ:perturbed_zero_order_prototype]{8} when $q = 0$. Then there exists an \(\epsilon > 0\) such that for any \(q \in \epsilon \mathbb{D} \subset \MR^d \),  these assumptions still hold for Problem~\hyperref[equ:perturbed_zero_order_prototype]{8} after perturbation by $q$.
\end{lemma}

\begin{proof}
    See Appendix \ref{sec: proof for stability}. 
\qed \end{proof}
%we can delete this if necessary
\begin{remark}
In Lemma \ref{lemma: perturbation_all_stability}, we choose \(q\) from a cube region instead of the more conventional ball region, because we later want to generate \(q\) via a uniform distribution, and the cube region ensures independence for each dimension of \(q\).

\end{remark}

Lemma \ref{lemma: perturbation_all_stability} establishes that Problem~\hyperref[equ:perturbed_zero_order_prototype]{8} follows the same assumptions as Problem~\hyperref[equ:zero_order_prototype]{4}. Therefore, all discussions in Section \ref{sec: Mechanism of LCvx} regarding Problem~\hyperref[equ:zero_order_prototype]{4} are also applicable to Problem~\hyperref[equ:perturbed_zero_order_prototype]{8} for sufficiently small $q$.
%%%%%%%%%%%%%%%%%%%%%%%%%%%%%

\subsection{Quantify the LCvx invalidation}
The goal for this section is to demonstrate that when the perturbation \(q\) is chosen randomly, the probability that \ac{LCvx} is invalid for problem~\hyperref[equ:perturbed_zero_order_prototype]{8} at no more than \(n_x-1\) points is zero. To begin, we examine the necessary conditions for Problem~\hyperref[equ:perturbed_zero_order_prototype]{8} to violate the original nonconvex constraints at the grid points indexed by \(P_1, P_2, \cdots, P_k \in \mathbb{N}\).

\begin{lemma}
\label{lemma: activate Srank}
For Problem~\hyperref[equ:perturbed_zero_order_prototype]{8} with a fixed perturbation $q$, we define 
$$
S = \left(\begin{array}{llll}
\tilde A(q)^{N-P_1}B & \tilde A(q)^{N-P_2} B & \cdots &\tilde A(q)^{N-P_k} B
\end{array}\right),
$$
where $P_1, P_2, \cdots, P_k \in \mathbb{N}$ are temporal grid points. Suppose that Assumptions \ref{assumption: uicompactness}, \ref{assumption: slater}, and \ref{assumption: eta_N nonzero} hold for Problem~\hyperref[equ:perturbed_zero_order_prototype]{8}. If \ac{LCvx} is invalid at the temporal grid points \( P_1, P_2, \cdots, P_k \in \mathbb{N} \), then the rank of \( S \) satisfies \( \operatorname{rank}(S) < n_x \).
\end{lemma}
\begin{proof}
According to Theorem \ref{lemma:LCvx lambda B},  the necessary condition for \ac{LCvx} to be violated at the $i$-th point is: $
\eta_i^\top B = 0 
$
Given that $\eta_i^\top =  \eta_N^\top \tilde A(q)^{N-i}$, it follows that $\eta_N^\top$ must lie in the left null space of $\tilde A(q)^{N-i}B$.  Consequently, if \ac{LCvx} is invalid at $\{P_1, P_2, \cdots, P_k\}$, $\eta_N^\top$ must lie in the left null space of the matrix $S$. According to Assumption \ref{assumption: eta_N nonzero}, $\eta_N$ is nonzero. Hence $S$ has nontrivial left null space. Then rank$(S)<n_x$.
\qed \end{proof}

%%%%%%%%%%%%%%%%%%%%%%%%%%%%%%%%
%%%Main theorem
%%%%%%%%%%%%%%%%%%%%%%%%%%%%%%%%

Now we have the main theorem.

\begin{theorem}
\label{theorem: main_theorem}
Given that Assumptions \ref{assumption: uicompactness}, \ref{assumption: dynamics full rank}, \ref{assumption: slater} and \ref{assumption: eta_N nonzero} hold for Problem~\hyperref[equ:zero_order_prototype]{4}, along with the controllability of \(\{A, B\}\). Then, there exists an \(\epsilon > 0\) such that if a vector \(q \in \MR^d\) is uniformly distributed over cube \(\epsilon \mathbb{D} \in \mathbb{R}^d\), the probability of the optimal trajectory for Problem~\hyperref[equ:perturbed_zero_order_prototype]{8} violating the nonconvex constraints in Problem~\hyperref[equ:original_LCvx]{3}  at more than \(n_x-1\) grid points is zero.
\end{theorem}

\begin{proof}   
See Appendix \ref{sec: proof for the main theorem}. The key to this proof is to use Lemma \ref{lemma: activate Srank}. We show that for a sufficiently small perturbation \(q\) and for arbitrary integers \(1 \leq P_1 < P_2 < \cdots < P_{n_x} < N+1\), where \(N\) is the total amount of temporal grid points, the matrix \(S\) defined in Lemma \ref{lemma: activate Srank} is of full rank with probability one.\qed \end{proof}

%%%%%%%%%%%%%%
% We can modify the equation to save space.
We also want to estimate the influence of the perturbation. Taking the control sequence that is optimum to Problem~\hyperref[equ:perturbed_zero_order_prototype]{8}, and applying it to the unperturbed dynamics, we are expecting that when perturbation $q$ is sufficiently small, the perturbation does not lead to a significant difference in the output of the unperturbed dynamics. 
\begin{theorem}
\label{thm: perturbation error control}
Suppose Assumptions \ref{assumption: uicompactness}, \ref{assumption: dynamics full rank}, \ref{assumption: slater} and  \ref{assumption: eta_N nonzero} hold for Problem~\hyperref[equ:perturbed_zero_order_prototype]{8}. Let the control sequence obtained by solving Problem~\hyperref[equ:perturbed_zero_order_prototype]{8} as \(u(q)\) and apply this control to the unperturbed dynamics, leading to a trajectory
\begin{equation}
\begin{aligned}
\bar{x}_{N+1}(q) & =x_1+\sum_{i=1}^N A^{i-1} B u_i(q) \\
& =x_1+\sum_{i=1}^N \bar{A}^{i-1}(0) B u_i(q)
\end{aligned}
\end{equation} where $u_i(q)$ is the control sequence $u(q)$ at the $i$-th temporal grid point. 
Then, for any \(\delta > 0\), there exists an \(\epsilon > 0\) such that for any \(\|q\| < \epsilon\), $\tilde x_{N+1}(q)$ and $u(q)$ satisfy:
 \begin{itemize}
    \item \(\|G(\tilde{x}_{N+1}(q))\| \leq \delta\),
    \item \(m(\tilde{x}_{N+1}(q)) + \sum_{i = 1}^N l_i(g(u_i(q))) \leq m^*_o + \delta\), 
 \end{itemize}
where \(m^*_o\) is the optimum value of the discretized nonconvex Problem~\hyperref[equ:original_LCvx]{3}.     
\end{theorem}

\begin{proof}
See Appendix \ref{sec: Proof for perturbation of error control}.
\qed \end{proof}

\begin{remark}
The perturbation mentioned in Theorem \ref{theorem: main_theorem} is introduced solely for mathematical discussion and is \textbf{not required} in practical applications. Nevertheless, it is mathematically necessary: the numerical example in Section \ref{sec:numerical}, constructed as a counterexample, demonstrates that LCvx can become invalid at more than \(n_x - 1\) nodes if no perturbation is applied, and it indicates that \(n_x - 1\) is a tight upper bound for the number of temporal grid points at which LCvx becomes invalid when the control dimension is one. \blue{In such deliberately constructed cases, any perturbation that meets the following two requirements is sufficient: it must be small enough to avoid significantly affecting the final state (as ensured by Theorem~\ref{thm: perturbation error control}) and large enough to be detected by the optimization algorithm. In our numerical implementation, we select a perturbation magnitude that is significantly larger than the solver’s effective numerical precision, ensuring that it meaningfully influences the solution. We then rerun the convex solver for the perturbed problem to obtain the final result.}
\end{remark}
 
%%%%%%%%%%%%%%%%%%%%%%%%%%%%%%%%%%%%%%%%%%%%%%%%%%%%%%%%%%%%%%%%%%
\blue{Our theoretical results guarantee at most $n_x - 1$ violations of control constraints, which may still be problematic in practice. To resolve this, we propose a simple correction step. For illustration, consider $g(x)=\|x\|$ (extensions are straightforward). Define the corrected control:
\[
\hat{u}_i = \begin{cases}
u_i, &\|u_i\|\geq \rho_{\min}\\[4pt]
\rho_{\min}\frac{u_i}{\|u_i\|}, &\text{otherwise}
\end{cases}
\]
and let $\{\hat{x}_i\}$ be the resulting state trajectory.}

\begin{theorem}\label{thm:violation_fixed}
\blue{Given a uniform discretization $t_f/N$, the final-state deviation satisfies:
\[
\|\hat{x}_{N+1}-x_{N+1}\|\leq (n_x-1)C_0\left(e^{\|A_c\|\frac{t_f}{N}}-1\right),
\]
where $C_0$ depends on $\{A_c,B_c,\rho_{\min}\}$. If $A_c$ is asymptotically stable, the bound tightens to:
\[
\|\hat{x}_{N+1}-x_{N+1}\|\leq C_1\frac{n_x-1}{N},
\]
where $C_1$ depends additionally on $t_f$.}
\end{theorem}

\begin{proof}
See Appendix \ref{sec: proof for violation_fixed}.
\end{proof}

\begin{remark}
\blue{Theorem~\ref{thm:violation_fixed} shows that the impact of violations decreases at a rate of $O(t_f/N)$, aligning with the continuous-time LCvx results as $N$ approaches infinity. Thus, violations become negligible for sufficiently large $N$.}

\blue{Practically, prior LCvx studies rarely encounter multiple violations, indicating that moderate discretization $N$ often suffices. Hence, one may select a sufficiently large $N$ based on theoretical bounds or adjust $N$ adaptively, starting from a moderate initial value and increasing it if violations significantly impact the results. Solutions obtained at smaller discretizations can serve as warm starts for larger problems, thereby minimizing computational cost.}
\end{remark}

%%%%%%%%%%%%%%%%%%%%%%%
\section{LCvx on Long-horizon Cases}
\label{sec: long-horizon}

In long-horizon cases, the optimal control inputs of Problem~\hyperref[equ:zero_order_prototype]{4} may violate the nonconvex constraints of Problem~\hyperref[equ:original_LCvx]{3} at all temporal grid points, i.e., \(g(u_i) < \rho_{\text{min}}\) for all \(i\). Our goal is to find an optimal solution to the \textit{continuous-time} nonconvex Problem~\hyperref[equ:unslacked_continuous]{1} that minimizes both the control cost and boundary cost. If such a solution exists, its final state \(x(t_f)\) should be the optimal solution to Problem~\hyperref[equ: boundaryonly]{7}, and the control input should satisfy \(g(u) = \rho_{\text{min}}\).

To solve the \textit{continuous-time} nonconvex Problem~\hyperref[equ:unslacked_continuous]{1}, we divide the trajectory into two phases, with \(t_s\) as the switching point. The first phase is defined on the interval \([0, t_s]\). In this phase, we set the control input \(u: [0, t_s] \rightarrow \mathbb{R}^{n_u}\) to be constant, ensuring that \(g(u) = \rho_{\text{min}}\), while the state trajectory follows the continuous dynamics described in Problem~\hyperref[equ:unslacked_continuous]{1}. The final state of the first phase then serves as the initial state for the second phase. In the second phase, we formulate a discrete-time convex optimization problem by discretizing the continuous dynamics over \([t_s, t_f]\). The boundary state constraint is maintained, and the convexified control constraints are evaluated at temporal grid points.

Determining the transition time $t_s$ is challenging, as it must be large enough for the second phase optimization problem to belong to the normal cases, yet small enough to reduce the control and final state costs. Hence, we adopt the bisection search method to determine $t_s$. We will demonstrate that the optimal total cost of the first and second phases is a continuous function of $t_s$, which justifies the use of the bisection search method. An optimal $t_s$ should minimize the total cost while ensuring that the discrete convex \ac{LCvx} problem in the second phase is normal.

As explored in Section \ref{sec: normal cases}, solving the discrete \ac{LCvx} problem does not ensure \ac{LCvx} validation at all temporal grid points. Thus, we aim to find a control sequence such that the constraint $\rho_{\text{min}} + \epsilon \geq g(u_i) \geq \rho_{\text{min}}$ for an arbitrarily small fixed $\epsilon > 0$, with violations at no more than $n_x - 1$ grid points after the perturbation described in Theorem \ref{theorem: main_theorem}. Additionally, $x_{N+1}$ should be an optimal solution to Problem~\hyperref[equ: boundaryonly]{7} up to any arbitrarily small fixed constant $\epsilon$. The mathematical definition for this statement is available in Theorem \ref{thm: over_relaxed main}.

%%%%%%%%%%%%%%%%%%%%%%%%%%%%%%%%%%%%%%%%%
Now, we rigorously define the bisection search method. In the first phase, we fix the control input as a constant $u_s \in \MR^{n_u}$ over a duration $t_s \in \MR$, with the initial condition $x_{s} \in \MR^{n_x}$ provided in the original problem. The control $u_s$ is chosen to ensure $g(u_s) = \rho_{\text{min}}$. The continuous dynamics matrices are denoted as $A_c$ and $B_c$. By plugging in $u_s$ into the continuous dynamics, we determine the final state of the first phase, which also serves as the initial condition for the second phase:
\begin{equation}
\label{equ: initial condition of the first phase}
    x_{\text{init}}(t_s) = \exp \left(A_c  t_s\right) x_s + \int_0^{t_s} \exp \left(A_c \tau\right) d \tau B_c u_s.
\end{equation} 

For the second phase, we define a discrete convexified optimization Problem~\hyperref[equ:zero_order_prototype]{4}, keeping the dynamics discretized as $N$ pieces with zero-order hold control. The discrete dynamics matrices are expressed as:
\begin{align}
A(t_s) &= \exp \left(A_c \frac{t_f-t_s}{N}\right), \\
B(t_s) &= \int_0^{{(t_f-t_s)}/{N}} \exp \left(A_c \tau\right) d \tau B_c,
\end{align}
where $t_f-t_s$ is the amount of time remaining for the second phase.

Furthermore, we modify the objective function of Problem~\hyperref[equ:zero_order_prototype]{4} by replacing \( l_i(\sigma_i) \) with \( l(\sigma_i) \times (t_f - t_s) / N \), where \( l \) is chosen such that its derivative on the domain of \( \sigma \) is lower-bounded by some \( L_l > 0 \). This adjustment simplifies the mathematical derivation. Theorem \ref{thm: over_relaxed main} justifies this modification, ensuring that the quality of the resulting trajectory and control variables remains unaffected.

The optimization problem for the second phase is then formulated as:
\newline
\textbf{Problem 9}:
\phantomsection
\label{equ:LCvxallmissed}
\begin{equation}
\begin{aligned}
& \underset{x, u, \sigma}{\min}
& & m(x_{N+1}) + \sum_{i = 1}^N  l(\sigma_i) \frac{t_f-t_s}{N} + t_s l(\rho_{\text{min}}) \\
&  \operatorname{s.t.} \quad 
& & x_{i+1} = A(t_s) x_{i} + B(t_s) u_i, \text{ for } i \in \intset 1 N . \\
&&& \rho_{\min } \leq \sigma_i \leq \rho_{\max },\ g\left(u_i\right) \leq \sigma_i,
\text{ for } i \in \intset 1 N . \\
&&& x_1 = x_{\text{init}}(t_s), \quad G(x_{N+1}) = 0. 
\end{aligned}
\end{equation}
Here, we introduce a constant term \( t_s l(\rho_{\text{min}}) \) to the objective function of the aforementioned optimization problem to account for the cost of the first phase. This addition simplifies the proof while leaving the optimal solution of the second phase unaffected. The value of the optimal objective function, denoted by \( v(t_s) \), depends on \( t_s \) and represents the total cost of the flight, where the minimum value of \( v(t_s) \) is \( m^* + t_f l(\rho_{\text{min}}) \), with \( m^* \) being the optimal value of Problem~\hyperref[equ: boundaryonly]{7}. Our primary objective is to determine a \( t_s \) such that \( v(t_s) < m^* + t_f l(\rho_{\text{min}}) + \epsilon \), for a given and arbitrarily small \( \epsilon > 0 \). Additionally, to ensure the applicability of the theorems from the previous section, we require that Problem~\hyperref[equ:LCvxallmissed]{9} satisfies the normal case condition.

To simplify the notation for Problem~\hyperref[equ:LCvxallmissed]{9}, we define $z = (x, u, \sigma)$, the object function as $M(z, t_s)$, and all affine equalities as $F(z, t_s) = 0$. The remaining inequality constraints are identical to the set $\tilde V$ in  Problem~\hyperref[equ:simplified_zero_order_hold]{5}. Thus, our problem can be represented as:
\newline
\textbf{Problem 10}:
\phantomsection
\label{equ:long-horizon_t}
\begin{equation}
\begin{aligned}
 \min_z \quad& M\left(z , t_s\right) \\
\operatorname{s.t.} \quad & F\left(z, t_s\right)=0, \quad z \in \tilde V.
\end{aligned}
\end{equation}
The optimal value is a function of $t_s$, denoted as $v(t_s)$. 

We require the following assumption for Problem~\hyperref[equ:long-horizon_t]{10}.

\begin{assumption}
\label{assumption: valuefun}
Problem~\hyperref[equ:zero_order_prototype]{4} belongs to the long-horizon case. Meanwhile, there exists a $t_b$ such that $v(t_b)>m^* + t_f l(\rho_{\text{min}})$. Assumptions \ref{assumption: uicompactness}, \ref{assumption:controllability}, \ref{assumption: dynamics full rank}, and \ref{assumption: slater} are valid for Problem~\hyperref[equ:long-horizon_t]{10} generated by $t_s \in [0, t_b]$. 
\end{assumption}

%If the first part of Assumption \ref{assumption: valuefun} is not true, then when $t_f = t_s,$ it implies that the control sequence of the first phase is the \ac{LCvx} optimal control, since it achieves an $x_{N+1}$ that ensures $m^*$ for Problem~\hyperref[equ: boundaryonly]{7} and satisfies $g(u_i) = \rho_{\text{min}}.$  

\begin{lemma}
 \label{lemma: second eta neq 0} 
Assumption \ref{assumption: valuefun} ensures that  Problem~\hyperref[equ:long-horizon_t]{10} generated by \(t_b\) has a non-zero \(\eta_N^*\).
 
\end{lemma}
\begin{proof}
The condition \(v(t_b) > m^* + t_f l(\rho_{\text{min}})\) ensures that either \(x_{N+1}\) is not the solution to Problem~\hyperref[equ: boundaryonly]{7}, or at least one \(\sigma_i^*\) exceeds \(\rho_{\text{min}}\). In both cases, \(\eta_N^* \neq 0\) follows by the contrapositive of Lemma \ref{thm: eta_n is zero}.\qed
\end{proof}

Consequently, we have the following continuity result:
\begin{lemma}
Assumption \ref{assumption: valuefun} ensures that $v(t)$ is continuous over $(0, t_b]$.  
\end{lemma}
\begin{proof}
Assumption \ref{assumption: valuefun} guarantees that, for each $t$, solutions exist for Problem~\hyperref[equ:long-horizon_t]{10}. For each $t$, we select one solution and denote it as $z^*(t)$. Consequently, we can apply Theorem \ref{thm: vt continuity} to Problem~\hyperref[equ:long-horizon_t]{10} at the point $(z^*(t), t)$ for all $t$, with the prerequisites of Theorem \ref{thm: vt continuity} being met by Assumption \ref{assumption: valuefun} and Lemma \ref{lemma: cQ} in the Appendix. The compactness in Theorem \ref{thm: vt continuity} is satisfied by the same method used in Lemma \ref{lemma: perturbation_normality_stability} in the Appendix..
\qed \end{proof}

Since $v(t)$ is a continuous function of $t$, we can apply the bisection search method to find the transition point. Under Assumption \ref{assumption: valuefun}, $v(0) = m^* + t_f l(\rho_{\text{min}}),$ and $v(t_b)> m^* + t_f l(\rho_{\text{min}}),$  where \( m^* \) is the optimal value of Problem~\hyperref[equ: boundaryonly]{7}.  We let $t_{\text{low}} = 0$ and $t_{\text{high}} = t_b$. In each iteration, we choose $t_{\text{mid}} = (t_{\text{low}} + t_{\text{high}} )/2$. If $v(t_{\text{mid}}) > m^* + t_f l(\rho_{\text{min}}),$ we replace $t_{\text{high}}$ with $t_{\text{mid}}.$ Otherwise, we replace $t_{\text{low}}$ with $t_{\text{mid}}.$ The following theorem is a direct consequence of a standard bisection search method argument, the fact that $v(t)$ is continuous and Lemma \ref{lemma: second eta neq 0}.

\begin{theorem}
\label{thm: bisection search method}
Under Assumption \ref{assumption: valuefun}, for any $\epsilon > 0$, the bisection search method finds a $t_s^* \in (0, t_b)$ ensuring the following inequality within finite iterations:
\begin{equation}
m^* + t_f l(\rho_{\text{min}}) < v(t_s^*) < m^* + t_f l(\rho_{\text{min}}) + \frac{\epsilon}{2},
\end{equation}
where \( m^* \) is the optimal value of Problem~\hyperref[equ: boundaryonly]{7}. 
Meanwhile, the Problem~\hyperref[equ:LCvxallmissed]{9} generated by $t_s^*$ is normal.
\end{theorem}

\begin{remark}
\blue{The complexity of the bisection search is logarithmic: given horizon \(T_f\) and precision \(\Delta t_{\min}\), the maximum number of iterations is \(N = \log_2(T_f/\Delta t_{\min})\). For example, \(T_f=10\,\text{min}\) and \(\Delta t_{\min}=1\,\text{s}\) yield \(N\approx 10\). In practice, the computational burden is often lower than expected for several reasons: (i) long-horizon cases typically occur in early stages, where the problem size and complexity are lower; (ii) early termination is possible once a satisfactory solution is found; and (iii) warm-starting from previous solutions is available in the MPC framework.}
\end{remark}

Since the Problem~\hyperref[equ:LCvxallmissed]{9} generated by $t_s^*$ is normal, we can now add the perturbation and apply the results from Theorem \ref{theorem: main_theorem} and \ref{thm: perturbation error control}. As tuning the switch point changes the horizon and discretization ofProblem~\hyperref[equ:LCvxallmissed]{9} , we cannot directly compare the solution ofProblem~\hyperref[equ:LCvxallmissed]{9}  to that of the discrete Problem~\hyperref[equ:original_LCvx]{3}. Therefore, we directly compare the result ofProblem~\hyperref[equ:LCvxallmissed]{9}  to that of the continuous-time Problem~\hyperref[equ:unslacked_continuous]{1}. Note that the best possible solution for Problem~\hyperref[equ:unslacked_continuous]{1} is the one with control $g(u(t)) = \rho_{\text{min}}$ for all $t$ and $m(x(t_f))$ equal to $m^*$, where $m^*$ is the optimal value of Problem~\hyperref[equ: boundaryonly]{7}. The next theorem shows that the solution ofProblem~\hyperref[equ:LCvxallmissed]{9}  with perturbation achieves $m(x(t_f))$ almost equal to $m^*$, and $g(u_i)$ less than $\rho_{\text{min}}$ up to an arbitrarily small $\epsilon>0.$

To do this, we need a few notations consistent with those of Theorem \ref{theorem: main_theorem}. Consider the Problem~\hyperref[equ:LCvxallmissed]{9} generated from $t_s^*$. We add a perturbation $q \in \MR^d$ to the dynamics matrix $A(t_s^*)$ and denote the perturbated matrix as $\tilde A(t_s^*, q)$. The optimal solution of the resulting perturbed optimization problem is denoted as $(x(q), u(q), \sigma(q))$. If we plug in $u(q)$ into the original dynamics, we have
\begin{equation}
\begin{aligned}
\tilde{x}_{N+1}(q) & =x_1+\sum_{i=1}^N A\left(t_s^*\right)^{i-1} B\left(t_s^*\right) u_i(q) \\
& =x_1+\sum_{i=1}^N \tilde{A}^{i-1}\left(t_s^*, 0\right) B\left(t_s^*\right) u_i(q)
\end{aligned}
\end{equation}
Here, $\tilde{x}_{N+1}(q)$ is denoted to be the endpoint of the propagated trajectory if we apply the control acquired from the perturbed problem to the original dynamics.

\begin{theorem}
\label{thm: over_relaxed main}
Assume Assumption \ref{assumption: valuefun} is valid, Choose a function $\Rfunc{l}{}{}$ for Problem~\hyperref[equ:LCvxallmissed]{9} such that the derivative of $l$ is larger than some $L_l>0$ within the domain of $\sigma$. Then, for any $\epsilon > 0$, we can acquire $t_s^* \in \MR$ and $\delta > 0$ such that for any perturbation $q\in \MR^d$ with $\|q\| < \delta$ applied to Problem~\hyperref[equ:LCvxallmissed]{9} generated by $t_s^*$, the control $u(q)$ and the propagated trajectory $\tilde{x}_{N+1}(q)$ satisfies:
\begin{itemize}
  \item $\|G(\tilde{x}_{N+1}(q))\| \leq \delta$,
  \item $g(u_i) \geq \rho_{\text{min}}$ for all $i$ up to $n_x-1$ points with probability one.
  \item $\rho_{\text{min}} + \epsilon > g(u_i)$ for all $i$,
  \item $m^* + \epsilon > m(\tilde{x}_{N+1}(q))$,
\end{itemize}
where $m^*$ is the solution to Problem~\hyperref[equ: boundaryonly]{7}.
\end{theorem}
%%%%%%%%%%

\begin{proof}
See Appendix \ref{sec: proof for over_relax}
\qed \end{proof}

%%%%%%%%%%%
Theorem \ref{thm: over_relaxed main} guarantees the validity of adding perturbation to the dynamics. Consequently, we introduce Algorithm \ref{algo}, which is capable of handling both normal and long-horizon cases. This algorithm provides a control sequence that at most violates the original nonconvex constraints at $n_x - 1$ grid points. In practical applications, the perturbation step is likely to remain deactivated since it serves primarily for theoretical completeness and is activated only in artificial edge cases.

%%%%
\begin{algorithm}
\caption{LCvx Method with Bisection for the Long Horizon Cases}
\label{algo}
\begin{algorithmic}[1]
\REQUIRE  $A_c$, $B_c$,$A$,$B$, $u_s$, $x_s$, $\rho_{\text{min}}$, $\epsilon_t > 0$, $\epsilon_{q}>0$
\ENSURE Optimized time $t_s$, state $x$, control $u$, and slack variables $\sigma$
\STATE Set $t_{\text{low}} = 0$, $t_{\text{high}} = t_f$, $t_s^* = 0$, $x_{\text{init}} = x_s$
\STATE Solve Problem~\hyperref[equ:zero_order_prototype]{4} to obtain $x$, $u$, $\sigma$
\IF{$\sigma_i > \rho_{\text{min}}$ for some $i$ or $x_{N+1}$ doesn't solve Problem~\hyperref[equ: boundaryonly]{7}}
\IF{the number of nodes that \ac{LCvx} being invalid $\leq n_x -1$ }
\RETURN $t_{s}^*$, $x$, $u$, $\sigma$

\ELSE
\STATE Perturb the optimization Problem~\hyperref[equ:zero_order_prototype]{4} with randomly generated $\|q\| \leq \epsilon_{q}$ \STATE Solve the perturbated Problem~\hyperref[equ:zero_order_prototype]{4} to obtain $x$, $u$, $\sigma$
\RETURN $t_{s}^*$, $x$, $u$, $\sigma$
\ENDIF
\ENDIF
\STATE Modify the object function $l_i$ to $l(\sigma_i) \times (t_f-t_s)/N$ that satisfies assumption in Theorem \ref{thm: over_relaxed main}
\STATE Set $t_s = (t_{\text{low}} + t_{\text{high}})/2$
\WHILE{$(t_{\text{high}} - t_{\text{low}}) > \epsilon_t$}
\STATE Solve Problem~\hyperref[equ:LCvxallmissed]{9} defined by $t_s$ and $u_s$ to obtain $x$, $u$, $\sigma$
\IF{$\sigma_i > \rho_{\text{min}}$ for some $i$ or $x_{N+1}$ doesn't solve Problem~\hyperref[equ: boundaryonly]{7}}
\STATE Set $t_{\text{low}}, t_{\text{high}} = t_{\text{low}}, t_s$
\ELSE
\STATE Set $t_{\text{low}}, t_{\text{high}} = t_s, t_{\text{high}}$
\ENDIF
\STATE $t_s = (t_{\text{low}}+ t_{\text{high}})/2$ 
\ENDWHILE
\STATE $t_s^* = t_{\text{high}}$

\STATE Solve Problem~\hyperref[equ:LCvxallmissed]{9} defined by $t_s^*$ to obtain $x$, $u$, $\sigma$.
\IF{the number of nodes that \ac{LCvx} being invalid $\leq n_x -1$ }
\RETURN $t_{s}^*$, $x$, $u$, $\sigma$

\ELSE
\STATE Perturb the optimization Problem~\hyperref[equ:LCvxallmissed]{9}  by $\|q\| \leq \epsilon_{q}$ to obtain $x,\ u,\ \sigma$
\RETURN $t_{s}^*$, $x$, $u$, $\sigma$
\ENDIF

\end{algorithmic}
\end{algorithm}

\section{Numerical Examples}
\label{sec:numerical}
In this section, we present three illustrative numerical examples. The first and third examples simulate Moon landing scenarios for the normal and long-horizon cases, respectively. The second example is an artificial case designed to demonstrate the necessity of perturbation and the tightness of the \( n_x - 1 \) bound in Theorem \ref{theorem: main_theorem}. All computations are performed using CVXPY \cite{diamond2016cvxpy}.

One can verify that Assumptions \ref{assumption: uicompactness}, \ref{assumption:controllability}, \ref{assumption: dynamics full rank}, \ref{assumption: slater}, and \ref{assumption: eta_N nonzero} hold for the first and second examples. However, directly testing Slater's Condition in Assumption \ref{assumption: valuefun} for the third example is challenging, as it requires verification for all \( t_s \in [0,t_b] \). Developing a method to test this assumption is left for future research. All other conditions in Assumption \ref{assumption: valuefun} are satisfied for the third example.

\subsection{Well-Conditioned Example}
\label{sec: well-conditioned example}
\blue{We consider a Moon landing problem where the objective is to guide the spacecraft from an initial state to a target state while satisfying thrust constraints. The state variable is a six-dimensional vector, where the first three components represent position and the last three components represent velocity. The spacecraft has a constant mass of \( \bar m = 1000 \)kg, an initial velocity of \( (0,3,-50) \)m/s, and an initial position of \( (0, 0, 5000) \)m. The thrust is constrained between a minimum and a maximum value, where the maximum thrust \( T_{\max} \) is 7500 N, and the actual thrust range is \( [0.3, 0.8] \) of the maximum thrust. The gravitational acceleration on the Moon is \( -1.62\ \text{m/s}^2 \) in the \( z \) axes. The target state is a low-speed descent at \( (0,0,100) \) m with velocity \( (0,0,-5) \) m/s, allowing a smooth transition to subsequent control phases. The values of these parameters are adjusted based on the data provided in \cite{sostaric2005powered}.}

\blue{This problem follows the same formulation as Problem 4, with the only difference being in the dynamics. The control variable represents the thrust-to-mass ratio. Since the final state is fixed, the terminal cost function \( m(x_{N+1}) \) is set to zero, and the fixed final state constraint is enforced through \( G(x_{N+1}) = 0 \).
 The cost function \( l(\sigma) \) is defined as the absolute value, minimizing control effort. The thrust constraints are given by \( \rho_{\min} = 0.3T_{\max}/\bar m \) and \( \rho_{\max} = 0.8T_{\max}/\bar m \), with a total flight time of \( t_f = 60 \) s.}

\blue{The key difference from Problem 4 is the inclusion of lunar gravity. The discrete-time dynamics are given by
\[
x_{i+1}=A x_i+B\left(u_i+g\right),
\]
where \( g = (0, 0, -1.62) \) represents lunar gravity. The matrices \( A \) and \( B \) are obtained from the continuous-time dynamics:
\begin{equation}
A_c=\begin{bmatrix}
0_{2\times2} & I_{2\times2} \\
0_{2\times2} & 0_{2\times2} 
\end{bmatrix}, \quad 
B_c=\begin{bmatrix}
0_{2\times2} \\ I_{2\times2}
\end{bmatrix},
\end{equation}
and discretized using \( N \) time steps via Equation~\eqref{equ: dicrete dyanmics matrix}. Compared to Problem 4, the dynamics include an additional perturbation term due to gravity, which does not affect the theoretical discussions in the previous section.}

The solution trajectory and control magnitudes for $N = 10$ are shown in Figure~\ref{fig:state_control_standard}. We observe that \ac{LCvx} is invalid at one temporal grid, indicating that the solution to Problem~\hyperref[equ:zero_order_prototype]{4} does not always satisfy the constraints of Problem~\hyperref[equ:original_LCvx]{3}, even for a well-conditioned problem. \blue{Additionally, we tested the extent of \ac{LCvx} violations under finer discretizations. With $N = 30, 50, 100, 300$, we observed one violation for $N = 50$ and $500$, and no violations for $N = 30$ and 300. This confirms that the number of violations does not depend on the discretization level, consistent with our theoretical results.}

\begin{figure}[t]
    \centering
    \begin{subfigure}[t]{0.45\linewidth}
        \centering
        \includegraphics[width=\linewidth,height=3cm]{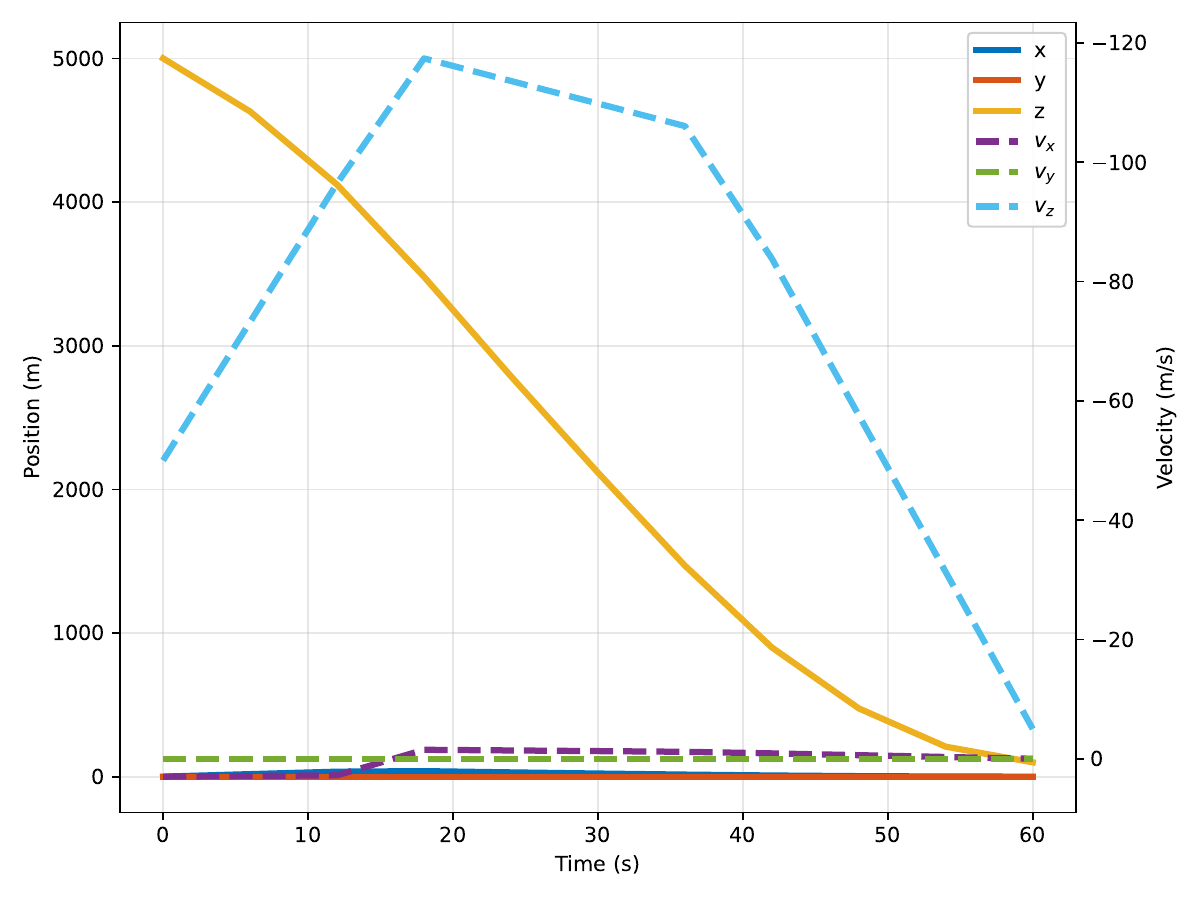}
        \caption{\scriptsize State Trajectory}
        \label{fig:state_trajectory}
    \end{subfigure}
    \hspace{0.5cm} 
    \begin{subfigure}[t]{0.45\linewidth}
        \centering
        \includegraphics[width=\linewidth,height=3cm]{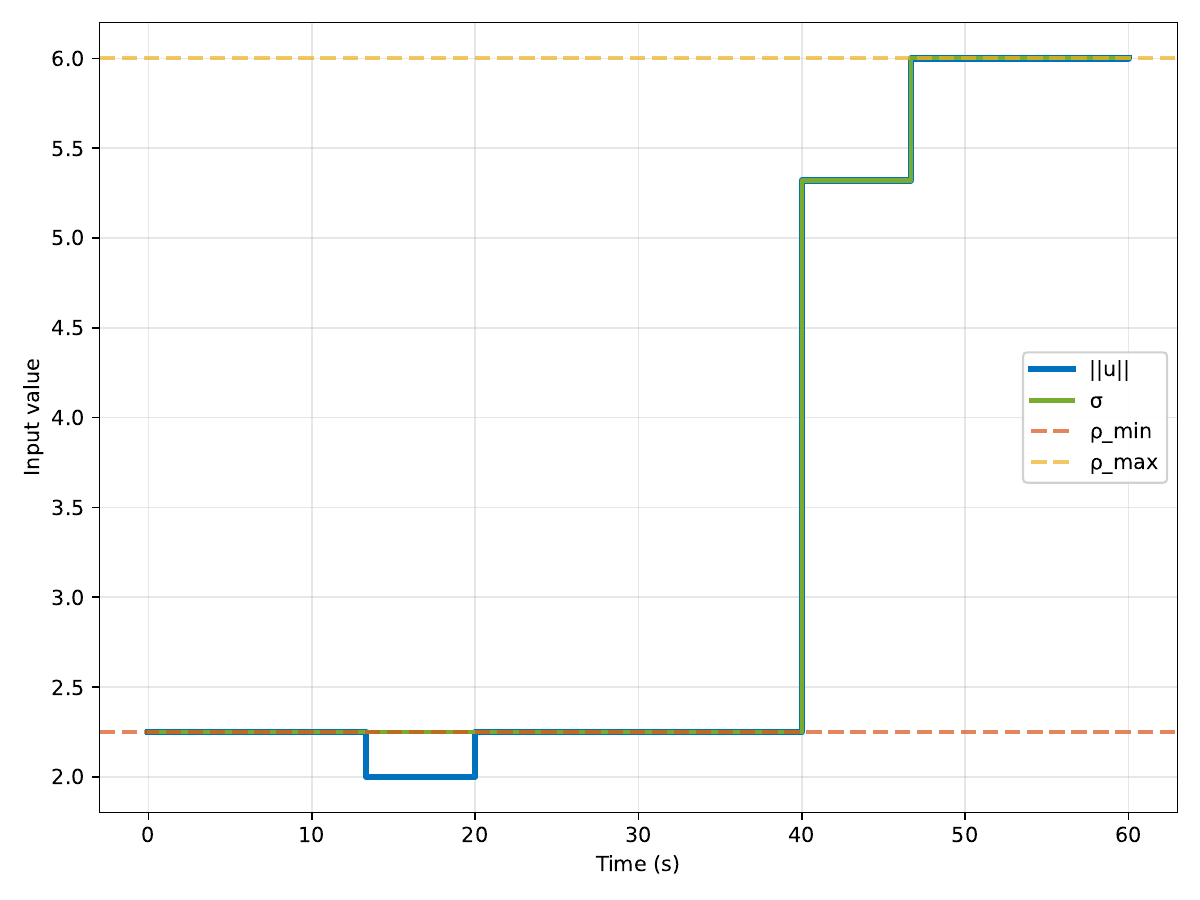}
        \caption{\scriptsize Control Magnitude}
        \label{fig:control_magnitude}
    \end{subfigure}
    \caption{\scriptsize State trajectory and control magnitude for the Moon landing problem over a 60-second flight. (a) Position and velocity evolution, with velocity plotted on an inverted axis. (b) Control norm, thrust bounds, and $\sigma$ evolution. One \ac{LCvx} invalid point occurs where the control norm falls below $\rho_{\min}$. }
    \label{fig:state_control_standard}
\end{figure}

\subsection{Artificial normal example}
\label{sec:Artificial normal example }
In this section, we consider a three-dimensional problem with $N = 10$, $A = \operatorname{diag}(1.2,-2.2,1)$, and $B = (0.4, 0.3, 0.2)^\top.$ Notice that $A$ and $B$ forms a controllable matrix pair. 
\blue{\begin{equation}
\begin{aligned}
& \min _{x, u, \sigma} \quad x_{11, 3}+\sum_{i=1}^{10} \sigma_i \\
& \operatorname{s.t.} \quad x_{i+1}=A x_{i}+ Bu_{i}, \quad i\in \intset {1}{10} \\
& x_{1}=[0,0,0]^{\top},\ x_{11,1}=0.5, \quad x_{11,2}=1 \\
& 1 \leq \sigma_i\leq 2, \quad \|u_i\|^2 \leq \sigma_i, \quad i\in \intset {1}{10} \\
\end{aligned}
\end{equation}
}
Here, the dynamics do not come from any continuous system, as $A$ has negative eigenvalues. This counterexample is formed by making the matrix $\Lambda$ in the proof of Theorem \ref{theorem: main_theorem} singular. The numerical tolerance for CVXPY is set to $10^{-12}$. The solution trajectory is available in Figure \ref{fig:state_before_perturb}, and the magnitude of the control is shown in Figure \ref{fig:control_before_perturb}. We observe that \ac{LCvx} is invalid at three points. According to Theorem \ref{theorem: main_theorem}, adding a perturbation to the dynamics reduces the violation of nonconvex constraints to at most two points. We consider the following dynamics with $A_p = \operatorname{diag}(10^{-7}, 0,0)$:
$$
x_{i+1}=
(A + A_p) x_{i}+ B u_{i}, \ i\in \intset {1}{10}
$$
which introduces a small change to one element in the dynamics matrix. The updated trajectory and control sequence are available in Figure \ref{fig:state_after_perturb} and Figure \ref{fig:control_after_perturb}, respectively. As expected, \ac{LCvx} is invalid at two points after perturbation.

\begin{figure}[tp]
\centering
\begin{subfigure}[t]{0.23\textwidth}
  \centering
  \includegraphics[width=\linewidth,height=3.5cm]{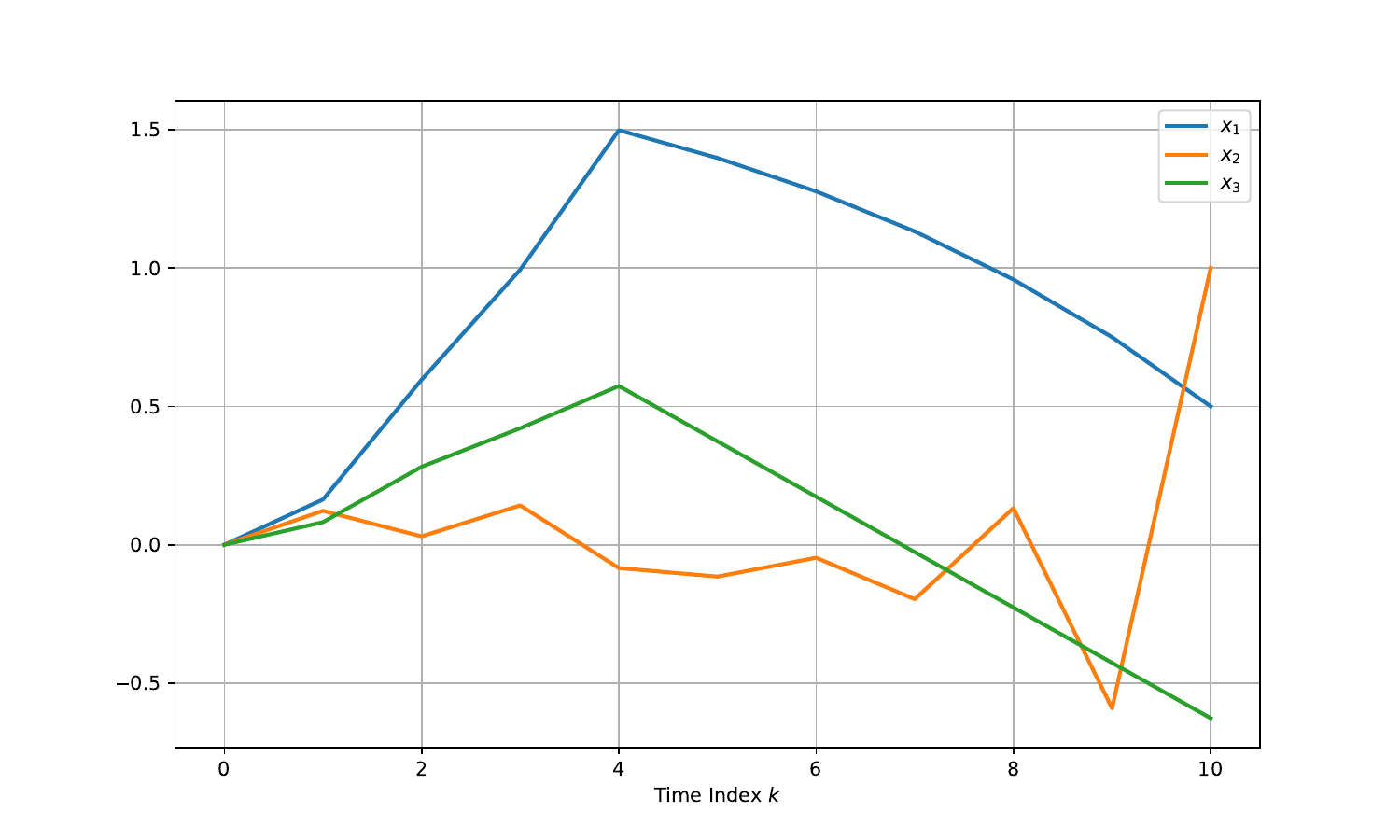}
  \caption{\scriptsize State before perturbation}
  \label{fig:state_before_perturb}
\end{subfigure}
\hfill
\begin{subfigure}[t]{0.23\textwidth}
  \centering
  \includegraphics[width=\linewidth,height=3.5cm]{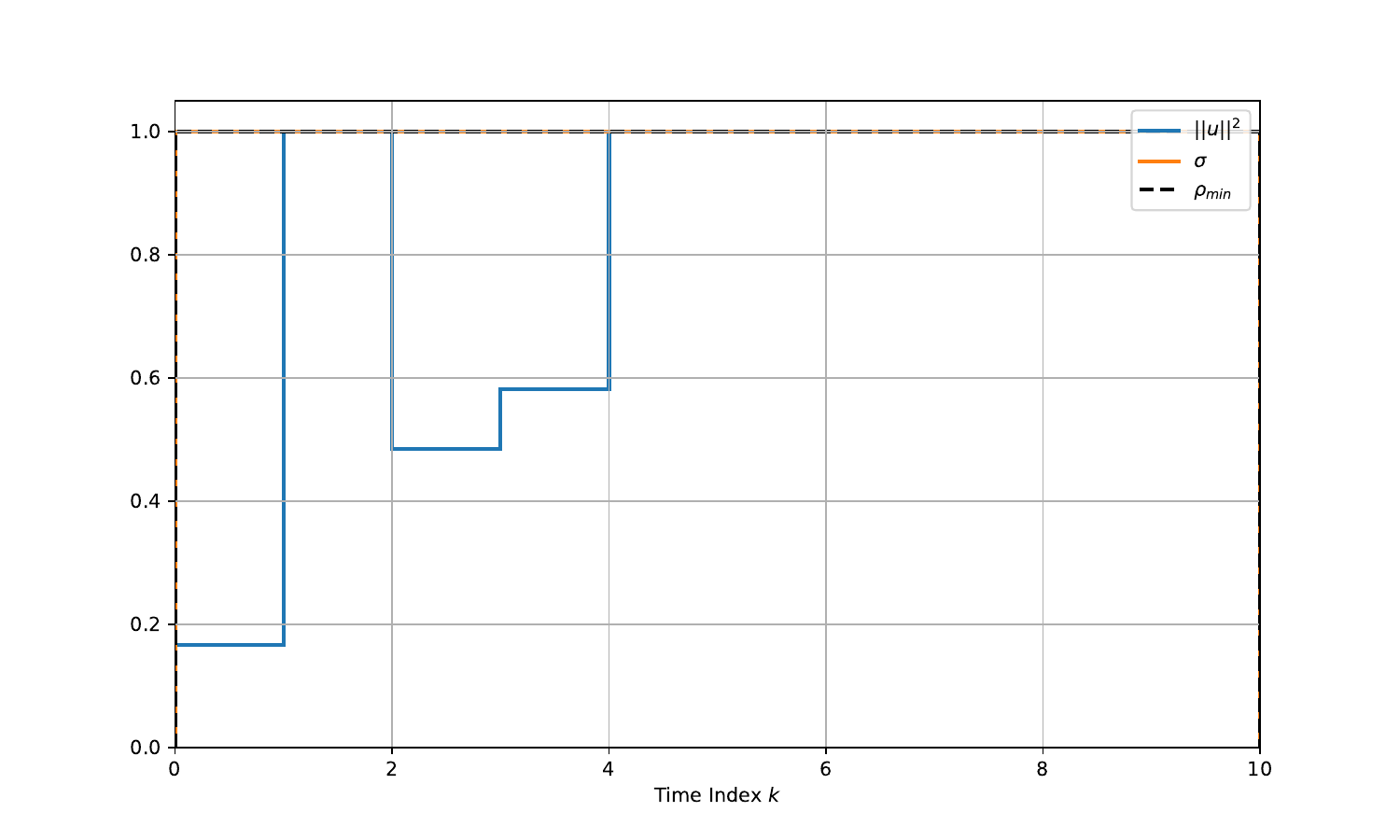}
  \caption{ \scriptsize Control before perturbation}
  \label{fig:control_before_perturb}
\end{subfigure}

\medskip

\begin{subfigure}[t]{0.23\textwidth}
  \centering
  \includegraphics[width=\linewidth,height=3.5cm]{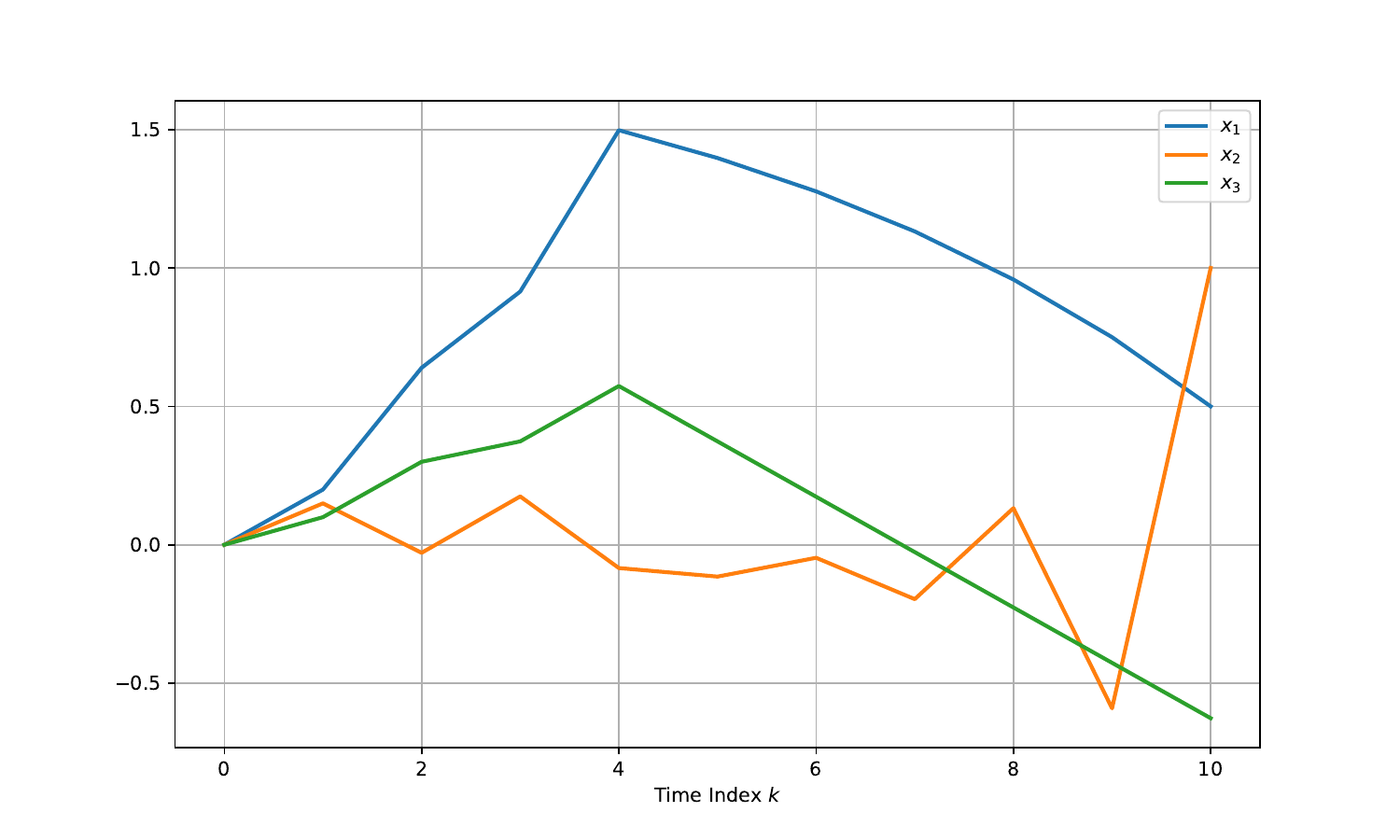}
  \caption{\scriptsize State after perturbation}
  \label{fig:state_after_perturb}
\end{subfigure}%
\hfill
\begin{subfigure}[t]{0.23\textwidth}
  \centering
  \includegraphics[width=\linewidth,height=3.5cm]{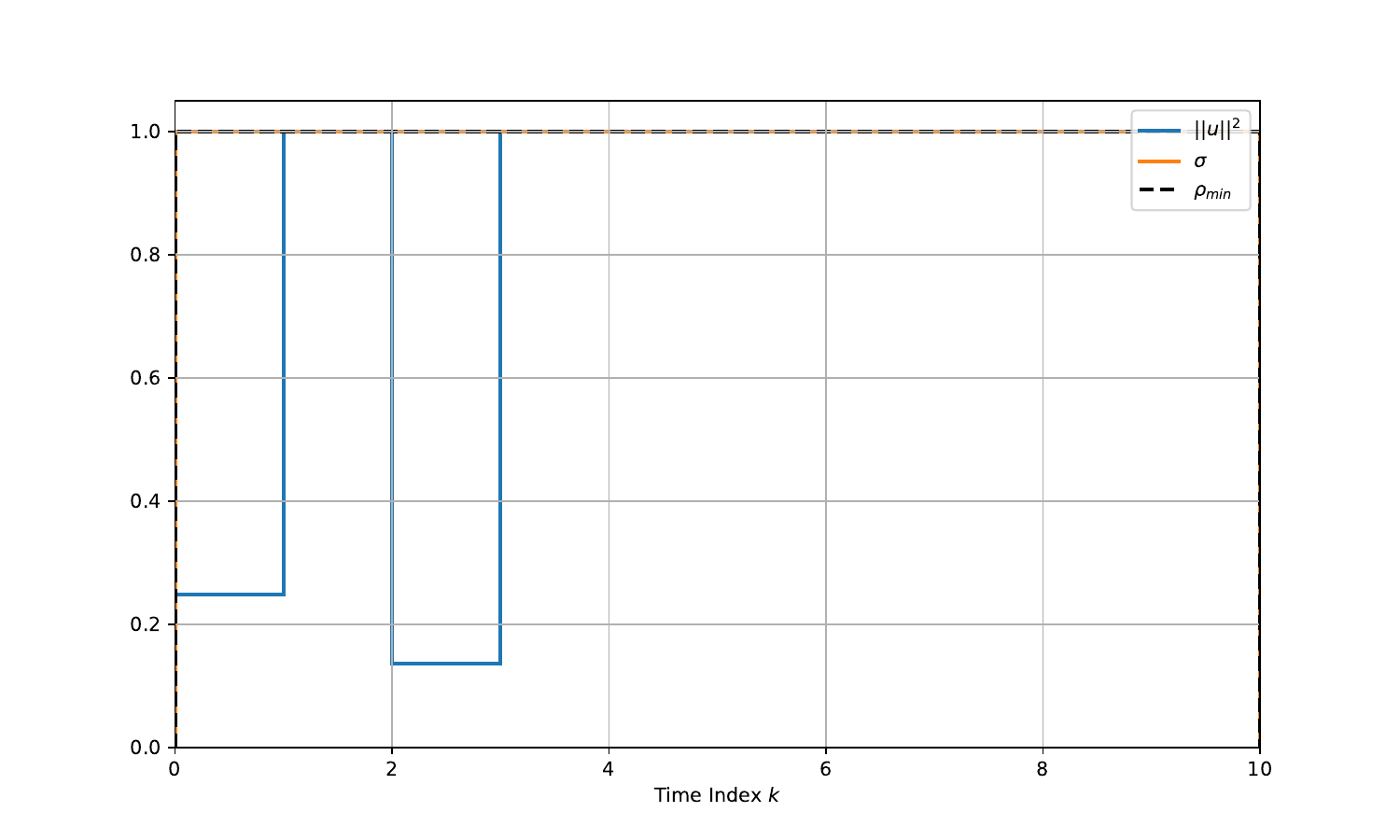}
  \caption{\scriptsize Control after perturbation}
  \label{fig:control_after_perturb}
\end{subfigure}
\caption{\scriptsize System dynamics and control response before and after perturbation for an artificial three-dimensional problem with \( N = 10 \) temporal grid points. The top figures display the state trajectory and control magnitude prior to perturbation, revealing \ac{LCvx} invalid at three points. The bottom figures show the system's response to a perturbation designed as per Theorem \ref{theorem: main_theorem}, which reduces the number of nonconvex constraint violations to two. The state variable remains essentially unchanged.}
\end{figure}

%%%%%%%%%%%%%%%%%%%%
\subsection{Long-horizon example}
\label{sec: long-horizon example}
Our final example is a long-horizon case that follows the same setup as our first example, with the distinction that we have \( N = 20 \) and \( t_f = 200 \). The trajectory and control magnitude are presented in Figures \ref{fig:over_relaxed_state_before} and \ref{fig:over_relaxed_control_before}. These figures show that all control magnitudes are less than \( \rho_{\text{min}} = 1 \), confirming that this is indeed a long-horizon case.

Next, we apply the bisection search method described in Algorithm \ref{algo}, setting \( \epsilon_t = 10^{-2} \) and initializing the control as \( [0,0,\rho_{\min}]^\top \). The algorithm terminates with \( t_{s} = 98.4 \) after seven bisection steps. This suggests maintaining the initial control for 98.4 seconds before switching to a control sequence generated by Problem~\hyperref[equ:LCvxallmissed]{9}. The resulting trajectory and control magnitude are shown in Figures \ref{fig:over_relaxed_state_after} and \ref{fig:over_relaxed_control_after}. The control magnitude plot reveals one \ac{LCvx} invalidation.

\begin{figure}[tp]
\centering

% First row: before bisection search method
\begin{subfigure}[t]{0.22\textwidth}
  \centering
  \includegraphics[width=\linewidth,height=3cm]{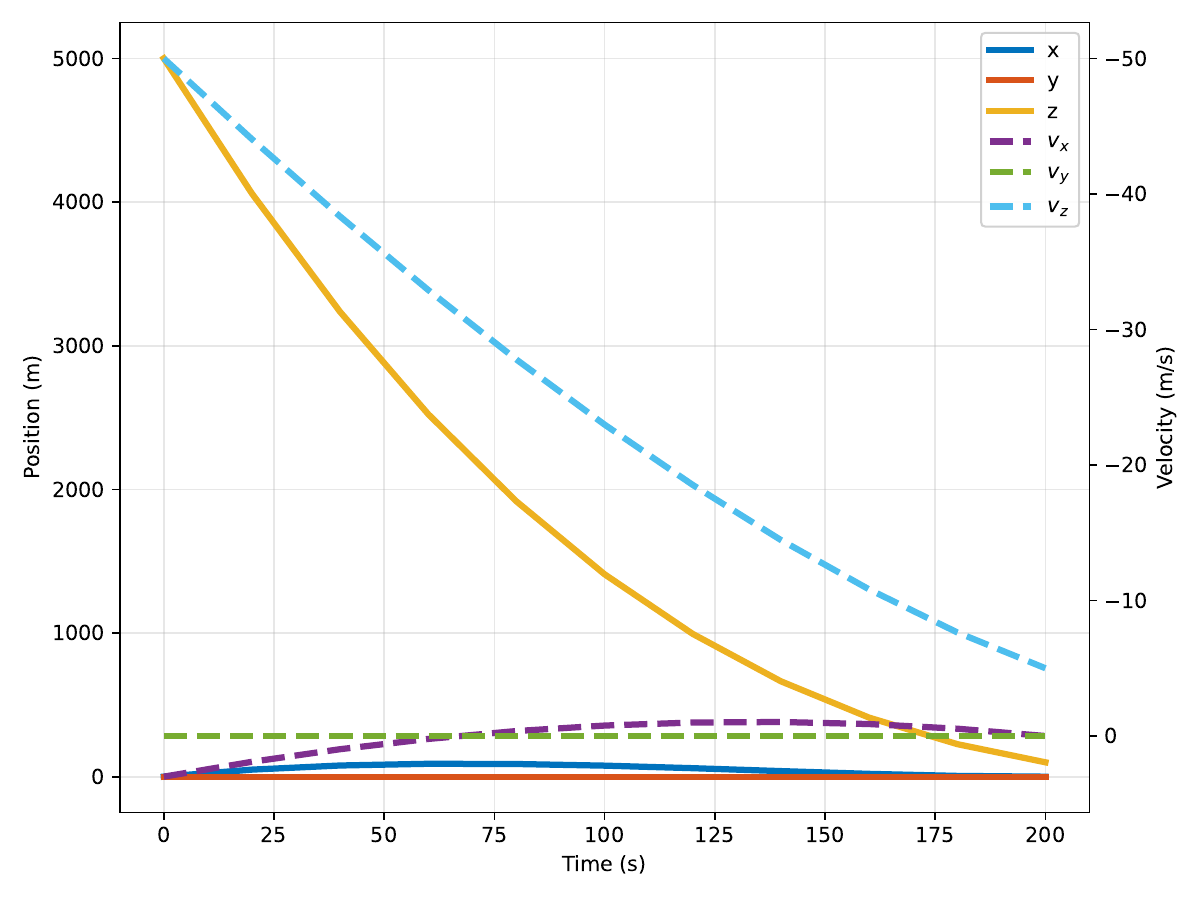}
  \caption{\scriptsize State before bisection}
  \label{fig:over_relaxed_state_before}
\end{subfigure}%
\hfill
\begin{subfigure}[t]{0.22\textwidth}
  \centering
  \includegraphics[width=\linewidth ,height=3cm]{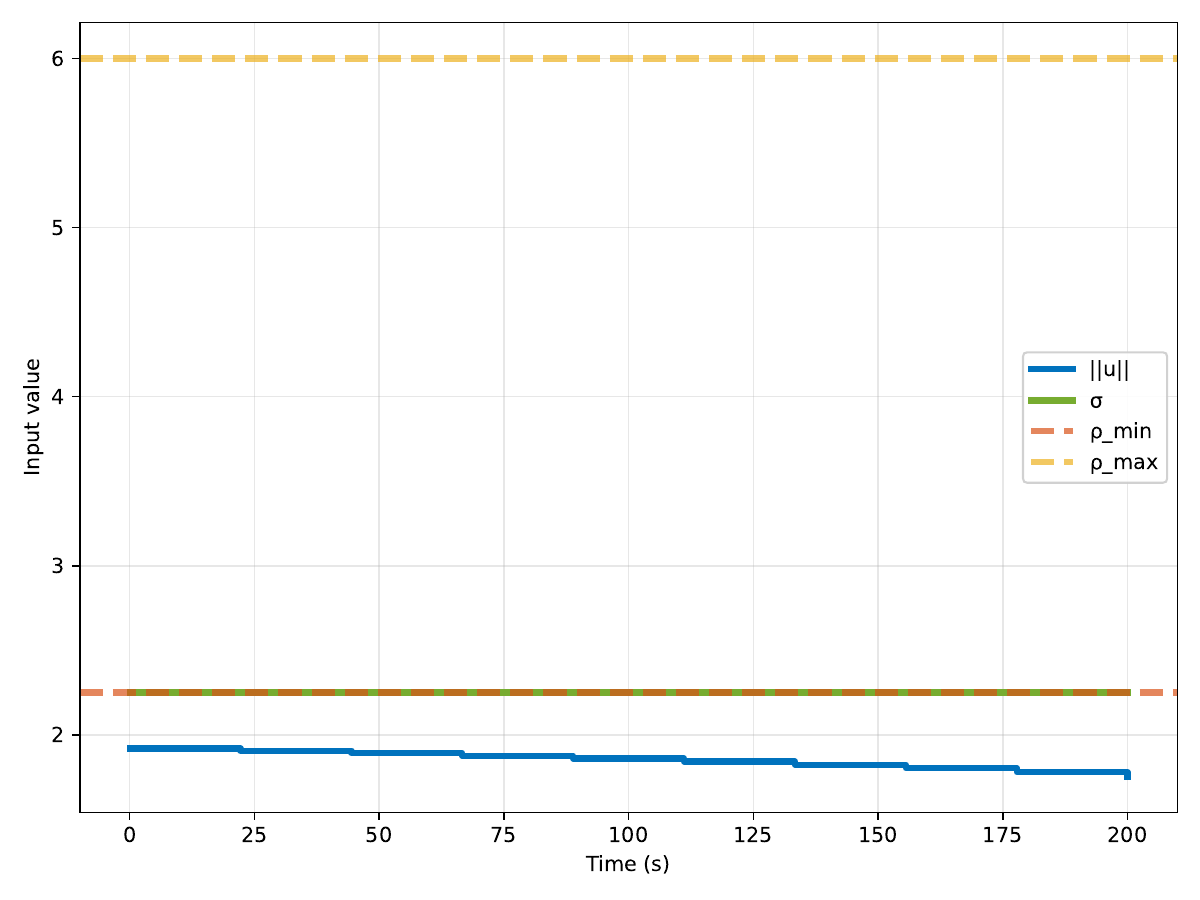}
  \caption{\scriptsize Control before bisection}
  \label{fig:over_relaxed_control_before}
\end{subfigure}

\medskip % Adds some vertical space

% Second row: after bisection search method
\begin{subfigure}[t]{0.22\textwidth}
  \centering
  \includegraphics[width=\linewidth,height=3cm]{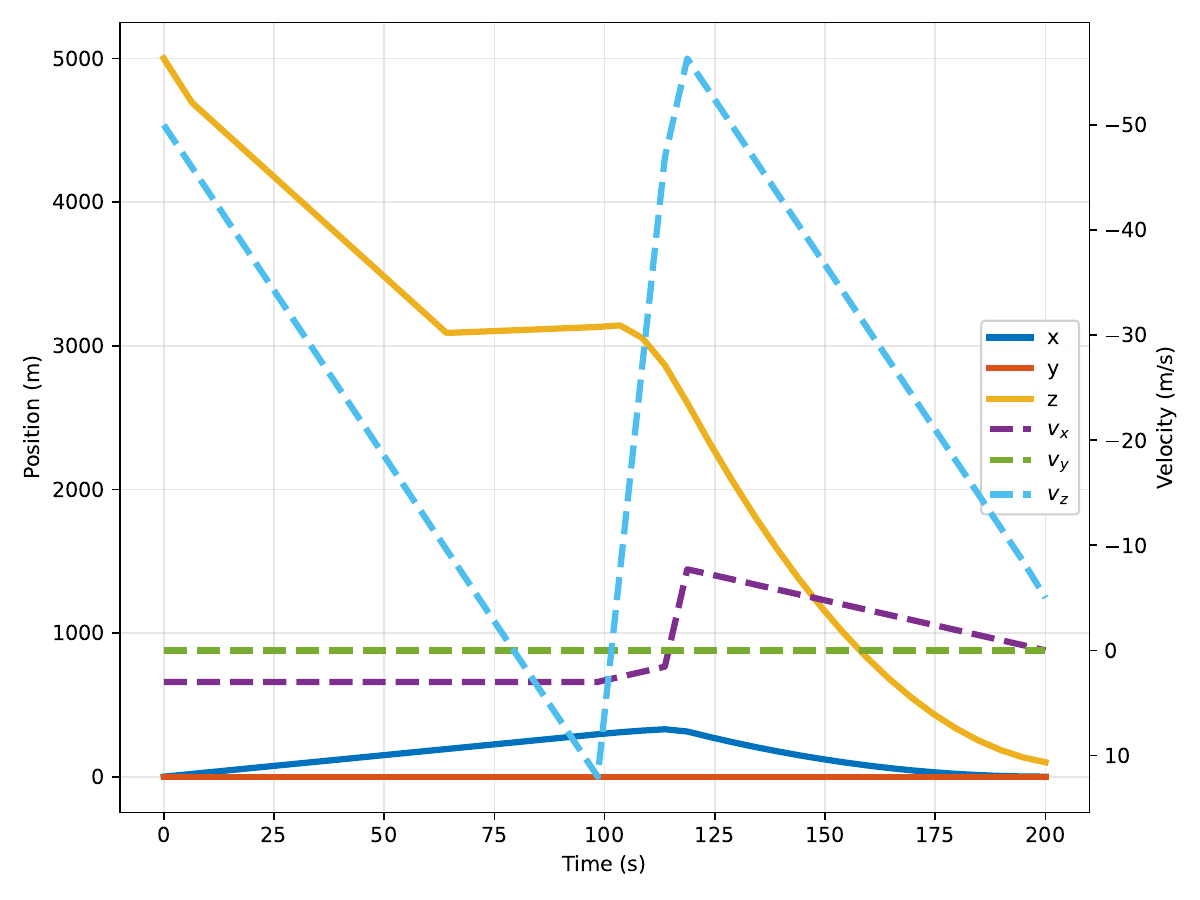}
  \caption{\scriptsize State after bisection}
  \label{fig:over_relaxed_state_after}
\end{subfigure}%
\hfill
\begin{subfigure}[t]{0.22\textwidth}
  \centering
  \includegraphics[width=\linewidth,height=3cm]{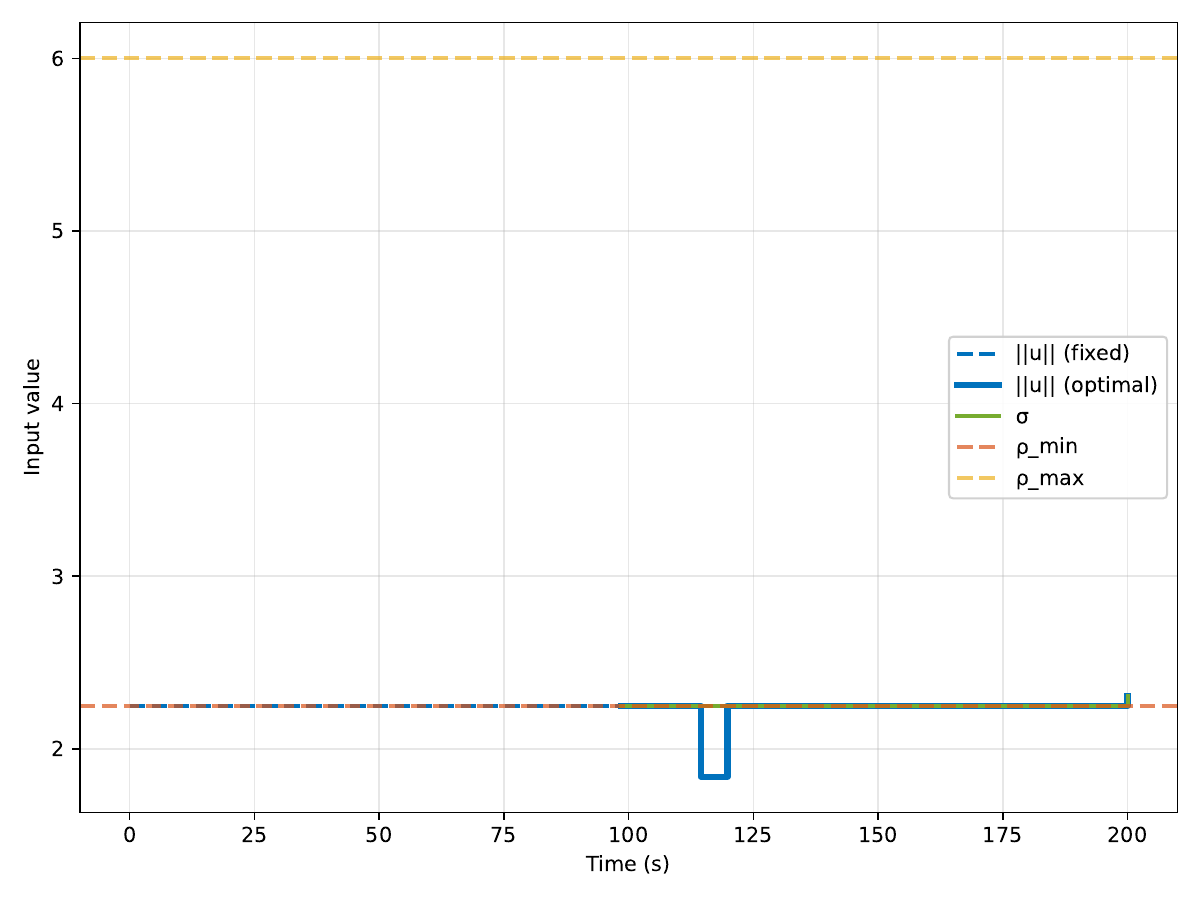}
  \caption{\scriptsize Control after bisection}
  \label{fig:over_relaxed_control_after}
\end{subfigure}

\caption{\scriptsize State trajectory and control magnitude for the Moon landing problem over a 200-second flight. (a,c) Position and velocity evolution before and after the bisection search, with velocity plotted on an inverted axis. (b,d) Control norm, thrust bounds, and \( \sigma \) evolution before and after the bisection search. Figures (a,b) indicate a long-horizon state where the control magnitudes remain below the threshold \( \rho_{\text{min}} \). Figures (c,d) demonstrate the effectiveness of the bisection search method described in Algorithm \ref{algo}, resulting in a trajectory that exhibits an \ac{LCvx} violation at one point. The information for \( \sigma \) is not available before the transition point.
}
\label{fig:over_relaxed_cases}
\end{figure}

\section{Conclusion}
In this paper, we develop a \blue{new theoretical framework for the application of} Lossless Convexification (LCvx) to discrete-time nonconvex optimal control problems, addressing a gap in the existing LCvx literature. Specifically, we establish theoretical guarantees that, under mild assumptions, the optimal solution of the relaxed convex problem satisfies the constraints of the original nonconvex problem at all but at most \(n_x - 1\) temporal grid points when an arbitrarily small perturbation is introduced into the dynamic constraints, where \(n_x\) is the state dimension. Furthermore, we analyze the limitations of the existing LCvx method in long-horizon cases where the transversality condition does not hold and demonstrate how our framework can be extended to handle these cases via a bisection-based approach. Numerical examples illustrate the necessity of the perturbation, the tightness of the \(n_x - 1\) upper bound, and the practical applicability of our theoretical framework to long-horizon problems.

For future work, we aim to investigate the theoretical guarantees of applying the \ac{LCvx} method to various forms of discrete-time nonconvex optimal control problems. Furthermore, we plan to extend our results to develop a lossless convexification method that ensures continuous-time constraint satisfaction~\citep{elango2024successive}.

\section{Appendix}

\subsection{Proof of theorem \ref{thm: vt continuity}}
\label{sec: proofs for background}
\begin{proof}    
Let \(y^*\) be the optimal solution to Problem~\hyperref[equ: general pertube]{6} at \(p_0\), satisfying assumptions in Lemma \ref{thm:clarkeexistance}. The existence of $y^*$ is a consequence of the compactness of $\Omega$ and the continuity of the object function. Lemma \ref{thm:clarkeexistance} ensures the existence of \(\delta > 0\) and \(\xi > 0\) such that for any \(p \in \MR^{n_p}\) with \(|p-p_0| < \xi\), there is a point \(y \in \Omega, F(y, p) = 0\) and 
\begin{equation}
\label{local_equ: contiuous proof}
\|y- y^* \| \leq \frac{\|F(y^*, p)\|}{\delta}.    
\end{equation}

Hence, the feasible set for Problem~\hyperref[equ: general pertube]{6} is nonempty for $p$ sufficiently close to $p_0.$ Additionally, because $\Omega$ is a compact set, optimal solutions exist for Problem~\hyperref[equ: general pertube]{6} for $p$ sufficiently close to $p_0.$

We first prove that \(\liminf_{s\rightarrow p_0} M^*(s) \geq M^*(p_0)\). Consider a sequence \(s_k \in \MR^{n_p}\) sufficiently close to $p_0$ such that \(\lim_k M^*(s_k) = \liminf_{s\rightarrow p_0} M^*(s)\) and \(\lim_k s_k = p_0\). We denote the optimal solution of the Problem~\hyperref[equ: general pertube]{6} generated by \(s_k\) as \(y_{k}\). Due to the compactness of $\Omega$, there exist a convergent subsequence \(y_{k_n}\) and a variable \(\bar{y}\) such that \(\lim_n y_{k_n} = \bar{y}\) and \(\liminf_{s\rightarrow p_0} M^*(s) = \lim_n M^*(s_{k_n}) = \lim_n M(y_{k_n}, s_{k_n})\). Due to the continuity of \(F(y, p)\) and the closedness of \(\Omega\), \(\bar{y}\) is feasible for Problem~\hyperref[equ: general pertube]{6} generated by \(p_0\). Thus, \(\lim_n M^*(s_{k_n}) = \lim_n M(y_{k_n}, s_{k_n}) = M(\bar{y}, p_0) \geq M^*(p_0)\) by the continuity of $M(\cdot, \cdot)$.

Next, we aim to show that \(\limsup_{w\rightarrow p_0} M^*(w) \leq M^*(p_0)\). As the object function \(M(y, p)\) is locally Lipschitz and \(y\) lies in a compact set, there exists a Lipschitz constant \(L_M\) for all feasible \((y, p)\) when \(\|p- p_0\|<\xi\).

Consider a sequence \(w_k\) such that \(\lim_{k} M^*(w_k) =  \limsup_{w\rightarrow p_0} M^*(w)\) with \(|w_k-p_0| < \xi\) and \(\lim_k w_k = p_0\). Lemma \ref{thm:clarkeexistance} ensures the existence of a corresponding sequence \(y_k \in \Omega\) that $F(y_k, w_k) = 0$,
\[
\|y_k- y^*\| \leq \frac{\|F(y^*, w_k)\|}{\delta},
\]
and \(M^*(w_k) \leq M(y_k,w_k)\).
Therefore,
\begin{align}
    \limsup_{w\rightarrow p_0} &M^*(w) = \lim_k M^*(w_k) \\
    &\leq \limsup_k M(y_k,w_k) \\
    &\leq \limsup_k \left( M(y^*, w_k) + L_M \|y_k - y^*\| \right)\label{equ: continuity explain 1} \\
    &\leq M(y^*, p_0) + L_M \limsup_k \frac{\|F(y^*, w_k)\|}{\delta}\\
    &= M(y^*, p_0) + \frac{\|F(y^*, p_0)\|}{\delta} = M^*(p_0). \label{equ: continuity explain 2}\\
\end{align} Here, inequality \eqref{equ: continuity explain 1} is due to the Lipschitz property of $M$ and inequality \eqref{equ: continuity explain 2} is due to the continuity of $F$.

In summary, we have
\[
\limsup_{p\rightarrow p_0} M^*(p) \leq M^*(p_0) \leq \liminf_{p\rightarrow p_0} M^*(p),
\]
thus establishing the continuity of \(M^*\). 
\qed \end{proof}

%%%%%%%%%%%%%%%%%%%%%%%%%
\subsection{Proof for Lemma \ref{lemma: perturbation_all_stability}}
\label{sec: proof for stability}
Since Assumption \ref{assumption: uicompactness} trivially holds, we start by showing the validity of Assumptions \ref{assumption:controllability} and \ref{assumption: dynamics full rank}.

\begin{lemma}
\label{lemma: perturbation_controllable_stability}
Suppose that Assumptions \ref{assumption:controllability} and \ref{assumption: dynamics full rank} hold for Problem~\hyperref[equ:perturbed_zero_order_prototype]{8} at $q = 0$. Then, there exists an \(\epsilon > 0\) such that for any \(q \in \MR^d \) within the cube \(\epsilon \mathbb{D} \in \MR^d \), where $\mathbb{D} \subset \MR^d$ is a unit cube, both assumptions are satisfied for Problem~\hyperref[equ:perturbed_zero_order_prototype]{8} generated by $q$. 
\end{lemma}
\begin{proof}
The proof of both assumptions relies on the fact that the determinant of a matrix is a continuous function of elements in that matrix. We demonstrate Assumption \ref{assumption:controllability} for Problem~\hyperref[equ:perturbed_zero_order_prototype]{8}, while the nonsingularity of $\nabla_z H$ in Assumption \ref{assumption: dynamics full rank} follows similarly and thus is omitted.

Define
\begin{equation}
\label{equ: matrix_S}
S(q) := \left(\begin{array}{llll}
B & \tilde{A}(q) B & \cdots & \tilde{A}(q)^{n_x-1} B
\end{array}\right).
\end{equation}
We aim to show that \(S(q)\) retains full rank for all \(q\) within \(\epsilon \mathbb{D}\). \(S(q)\) has full rank if and only if the determinant of \(S(q) S(q)^\top\) is nonzero. As the elements in matrix \(\tilde{A}\) are continuous functions of \(q\), the elements and, consequently, the determinant of \(S(q) S(q)^\top\) are continuous functions of \(q\). Given that \(\tilde{A}(0) = A\), the controllability of the matrix pair \(\{A,B\}\) ensures that \(S(0)\) has full rank and that the determinant of \(S(0) S(0)^\top\) is nonzero. Through continuity, an \(\epsilon > 0\) exists such that for any \(q\) within \(\epsilon \mathbb{D}\), the determinant of \(S(q) S(q)^\top\) remains nonzero, and \(S(q)\) maintains full rank.
%The existence of \(\epsilon\) is a corollary of~\citep[Theorem 4.8]{rudin1964principles}.
\qed \end{proof}

Now, we turn to prove that Assumption \ref{assumption: slater} holds after perturbation. To simplify the notation for Problem~\hyperref[equ:perturbed_zero_order_prototype]{8}, we let \(z = (x, u, \sigma)\). The object function \(\Phi(z)\) and the set \(\tilde V\) formed by inequalities are the same as in  Problem~\hyperref[equ:simplified_zero_order_hold]{5}. The affine constraint function \(H(z)\) in  Problem~\hyperref[equ:simplified_zero_order_hold]{5} is augmented to \(\tilde H(z, q)\) to accommodate the dynamics and boundary constraints of Problem~\hyperref[equ:perturbed_zero_order_prototype]{8}. Thus, our perturbed problem can be represented as:
% \begin{equation}
% \label{equ:simplified_perturbation}
% \begin{array}{rl}
% \min_z & \Phi(z) \\
% \operatorname{s.t.}& \tilde H(z, q) = 0, \quad z \in \tilde V,
% \end{array}
% \end{equation}
\begin{equation}
\label{equ:simplified_perturbation}
\begin{aligned}
\min _z\quad & \Phi(z) \\
\text { s.t. }\; & \tilde{H}(z, q)=0, \quad z \in \tilde{V},
\end{aligned}
\end{equation}
When \(q = 0\), we have \(\tilde H(z, 0) = H(z)\), thereby retrieving the original Problem~\hyperref[equ:simplified_zero_order_hold]{5}.

Since we need lemma \ref{thm:clarkeexistance}  and theorem \ref{thm: vt continuity} for our discussion, we must demonstrate that their assumptions are valid.

\begin{lemma}
\label{lemma: cQ}
 Assumption \ref{assumption: dynamics full rank} and Assumption \ref{assumption: slater} for Problem~\hyperref[equ:zero_order_prototype]{4} ensures that for any vector $z_0$ such that 
 $$
  \tilde H(z_0, 0) = 0, \quad z_0 \in \tilde V,
 $$
we have that $y = 0$ is the unique solution of the following equation:
\begin{equation}
0 \in \nabla_z \tilde H\left(z_0, 0\right)^\top y + N_{\tilde V}\left(z_0\right) 
\end{equation}

\end{lemma}

\begin{proof}
We prove this statement by contradiction. Assume there exists \(y \neq 0\) such that \[0 \in \nabla_z \tilde H\left(z_0, 0\right)^\top y + N_{\tilde V}\left(z_0\right).\] Consider the function \(g(z) = y^\top \tilde H(z, 0) + I_{\tilde V}(z)\), where \(I_{\tilde V}(z)\) is the indicator function for the set \(\tilde V\). Thus we have \(0 \in \partial g(z_0)\) and \(\tilde H(z_0, 0) = 0\). Since \(g(z)\) is a convex function, the minimum of \(g(z)\) is zero, achieved at $z_0$. 

Given Assumption \ref{assumption: slater}, there exists an interior point \(z'\) in \(\tilde V\) where \(\tilde H(z', 0) = 0\) and \(N_{\tilde V}(z') = \{0\}\). Therefore, \(z'\) is also a minimizer of the function \(g\), and thus \(0 \in \partial g(z') = y^\top \nabla_z \tilde H(z', 0)\). Meanwhile, $\nabla_z \tilde H(z', 0) = \nabla_z \tilde H(z_0, 0)$ because $H$ is an affine function of variable $z$. Consequently, the vector \(y\) lies in the left null space of \(\nabla_z \tilde H(z,0)\), which contradicts Assumption \ref{assumption: dynamics full rank}.
\qed \end{proof}

Now, we are ready to prove that Assumption \ref{assumption: slater} holds after perturbation:

\begin{lemma}
\label{lemma: perturbation_slater_stability}
Suppose that Assumption \ref{assumption: dynamics full rank} and \ref{assumption: slater} hold for Problem~\hyperref[equ:perturbed_zero_order_prototype]{8} when \(q = 0\). Then, there exists an \(\epsilon > 0\) such that for any \(q \in \MR^d \) within the cube \(\epsilon \mathbb{D} \in \MR^d \), where $\mathbb{D} \in \MR^d$ is an unit cube, such that the perturbed Problem~\hyperref[equ:perturbed_zero_order_prototype]{8} with respect to $q$ also satisfies Assumption \ref{assumption: slater}.
\end{lemma}
\begin{proof}
Slater's assumption implies the existence of a \(z_0\) such that \(H(z_0)= \tilde H(z_0,0) = 0\), while \(z_0\) lies in the interior of \(\tilde V\). Since Lemma \ref{lemma: cQ} ensures the applicability of Lemma \ref{thm:clarkeexistance}, there exist \(\epsilon_1 > 0\) and \(\delta > 0\) such that for all \(\|q\| < \epsilon_1\), there exists a \(z(q)\), determined by \(q\), satisfying 
\[
\|z(q) - z_0\| \leq \frac{\|\tilde H(z_0, q)\|}{\delta},
\]
and \(\tilde H(z(q), q) = 0\), thus meeting the dynamics and boundary conditions of Problem~\hyperref[equ:perturbed_zero_order_prototype]{8}. Given that $\tilde H$ is a continuous function, there exists an $\epsilon > 0$, such that when  \(\|q\|<\epsilon\), $\|z(q) - z_0\|$ is sufficiently small for \(z(q)\) to lie in the interior of \(\tilde V\), thereby validating Slater's condition for the perturbed problem. \qed

 \end{proof}

%%%%%%%%%%%%%%%%%%%%%%%%%%%%%%%%
Now we show that the perturbed problem is normal.
\begin{lemma}
\label{lemma: perturbation_normality_stability}
Suppose that Assumptions \ref{assumption: uicompactness}, \ref{assumption: dynamics full rank}, \ref{assumption: slater} and  \ref{assumption: eta_N nonzero} hold for Problem~\hyperref[equ:perturbed_zero_order_prototype]{8} and $q = 0$. Then there exists an \(\epsilon > 0\) such that for any \(q\) within the cube \(\epsilon \mathbb{D}\), Assumption \ref{assumption: eta_N nonzero} is still valid for Problem~\hyperref[equ:perturbed_zero_order_prototype]{8} after perturbation by $q$.
\end{lemma}
\begin{proof}
We apply Theorem \ref{thm: vt continuity} to show that the optimum value of Problem~\hyperref[equ:perturbed_zero_order_prototype]{8} is continuous with respect to \(q\). To satisfy the theorem's assumptions, we introduce a sufficiently large trivial upper bound on \(\sigma\), ensuring the compactness of \(\tilde V\). Assumption \ref{assumption: uicompactness} guarantees that the optimal solution remains unchanged with or without this upper bound. Consequently, Assumption \ref{assumption: uicompactness} and Lemma \ref{lemma: cQ} confirm the validity of assumptions to Theorem \ref{thm: vt continuity}.

Since when \(q = 0\) Problem~\hyperref[equ:zero_order_prototype]{4} is normal, the optimal value of Problem~\hyperref[equ:zero_order_prototype]{4} is larger than \(m^* + \sum_{i = 1}^N l_i(\rho_{\min})\), where $m^*$ is the minimum value of Problem~\hyperref[equ: boundaryonly]{7}. Thus, the continuity of the optimal value function ensures the existence of an \(\epsilon > 0\) such that for all \(q \in \epsilon \mathbb{D}\), the optimal value of Problem~\hyperref[equ:perturbed_zero_order_prototype]{8} generated by \(\{\tilde A(q), B\}\) is greater than \(m^* + \sum_{i = 1}^N l_i(\rho_{\min})\). Assumption \ref{assumption: uicompactness} ensures that $l_i$ are monotone, indicating that the perturbed problem is still normal.
\qed \end{proof}
\subsection{Proof for Theorem \ref{theorem: main_theorem}}
\label{sec: proof for the main theorem}
The following is a useful linear algebra theorem from~\citep[Theorem 3.5]{chen1984linear} that aids us in proving Theorem \ref{theorem: main_theorem}.

\begin{theorem}
\label{theorem: matrix poly}
Consider a function $f(\lambda)$ and a matrix $A \in \MR^{n \times n}$. The characteristic polynomial of $A$ is
\[
\Delta(\lambda) = \operatorname{det}(I - \lambda A) = \prod_{i=1}^d (\lambda - \lambda_i)^{m_i},
\]
where $m_i$ are the algebraic multiplicity of eigenvalue $\lambda_i$, $n=\sum_{i=1}^d m_i$, and $d$ is the total number of distinct eigenvalues.

There exists a polynomial $h(\lambda)$ as
\begin{equation}
\label{equ: matrix poly}
    h(\lambda) := v_0 + v_1 \lambda + \cdots + v_{n-1} \lambda^{n-1},
\end{equation}
such that 
\[
f(A) = h(A).
\]
The unknown coefficients of \eqref{equ: matrix poly} are found by solving $n$ equations:
\begin{equation}
\label{equJ: matrix poly coef}
f^{(l)}(\lambda_i) = h^{(l)}(\lambda_i) \text{ for } l\in \intset{0}{m_i-1}, i \in \intset{1}{d},
\end{equation}
where
\begin{equation}
f^{(l)}(\lambda_i) := \left. \frac{d^l f(\lambda)}{d \lambda^l} \right|_{\lambda=\lambda_i},
\end{equation}
and $h^{(l)}(\lambda_i)$ is defined in the same manner.

\end{theorem}

Now we start the proof for \ref{theorem: main_theorem}.
\begin{proof}[Proof for Theorem \ref{theorem: main_theorem} ]
If \( n_x > N \), the statement holds trivially; thus, we assume \( n_x \leq N \).
 We choose \(\epsilon > 0\) sufficiently small that Lemma \ref{lemma: perturbation_all_stability} is valid. Thus, the perturbated Problem~\hyperref[equ:perturbed_zero_order_prototype]{8} has controllable dynamics, is normal, and satisfies Slater's condition. 

Consider the matrix
\begin{equation}
S(q) := \left(\begin{array}{llll}
 \tilde{A}(q)^{P_1}B & \tilde{A}(q)^{P_2} B & \cdots & \tilde{A}(q)^{P_{n_x}} B
\end{array}\right),
\end{equation}
for specific integers \(1 \leq P_1 < P_2 < \cdots < P_{n_x} < N+1\), where \(N\) is the total number of grid points in Problem~\hyperref[equ:perturbed_zero_order_prototype]{8}, $\tilde A(q)$ is the perturbation of $A$ with respect to $q$ and \(n_x\) is the state dimension.

According to Lemma \ref{lemma: activate Srank}, if \(S(q)\) is of full rank, then the Problem~\hyperref[equ:perturbed_zero_order_prototype]{8} generated by \(\{\tilde{A}(q), B\}\) satisfies \ac{LCvx} at least at one point among the set \(\left\{N-P_1, N-P_2, \cdots, N-P_{n_x}\right\}\).

We aim to demonstrate that \(S(q)\) is of full rank with probability one, regardless of the choice of \(P_1\) to \(P_{n_x}\).

According to Theorem \ref{theorem: matrix poly}, all powers of the matrix \(\tilde{A}(q)\) can be written as polynomials of \(\tilde A(q)\) up to the power \(n_x-1\). Then, there exists a matrix \(V(q) \in \MR^{n_x \times n_x}\) such that
\begin{equation}
\tilde{A}^{P_i}(q)=\sum_{j=1}^{n_x} V_{i j}(q) \tilde{A}^{j-1}(q), \quad i \in [1,n_x]_\mathbb{N}
\end{equation}
Here, \(V_{ij}\) represents the element in the \(i\)-th row and \(j\)-th column of \(V(q)\).

Now we show that $V(q)$ being invertible is sufficient for $S(q)$ to be of full rank: Let $W(q)$ be the inverse of $V(q),$ Then 
\begin{equation}
\label{equ: W to controllable}
\tilde{A}^i(q)=\sum_{j=1}^{n x} W_{i j}(q) \tilde{A}^{P_i}(q),\quad i \in [1,n_x]_\mathbb{N}
\end{equation}
If $S(q)$ has a nonzero left null vector $\xi$, we have
$
\xi \tilde{A}^{P_i}(q) B=0 \text { for all }
$ for all $i\in \intset 1 {n_x}.$
Equation \eqref{equ: W to controllable} implies that
$
\xi \tilde{A}^{i}(q) B=0,
$ for all $i\in \intset 0 {n_x-1}$
contradicting controllability ensured by Lemma \ref{lemma: perturbation_all_stability}.

Therefore, $V(q)$ being invertible is sufficient for $S(q)$ to be of full rank. We thus have 
$$
\operatorname{Prob}( \operatorname{rank}(S(q))<n_x) \leq   \operatorname{Prob}(V(q) \text{ being singular})
$$
The rest of this proof is to show that the probability on the right-hand side is zero. 

Setting the distinct eigenvalues of $A$ as $\lambda_1, \lambda_2, \cdots, \lambda_d$, the eigenvalues of the perturbed matrix $\tilde A(q)$ have eigenvalues \(\tilde{\lambda}_i = \lambda_i + q_i\) for $i = \{1, 2, \cdots, d\}$. Here, $d$ is the number of distinct eigenvalues of $A$.

We start by exploring how \(V(q)\) is determined by the eigenvalues \(\tilde{\lambda}_i = \lambda_i + q_i\). The characteristic polynomial of \(A\) can be written as:
$
\prod_{i=1}^d (\lambda - \lambda_i)^{m_i}
$ where $\sum_{i=1}^d m_i=n_x$ with \(m_i\) being the algebraic multiplicity of \(\lambda_i\).

Due to the perturbation preserving the Jordan form, the characteristic polynomial of \(\tilde{A}\) becomes:
$$
\prod_{i=1}^d (\lambda - \tilde{\lambda}_i)^{m_i} = \prod_{i=1}^d (\lambda - \lambda_i - q_i)^{m_i},
$$
where $\sum_{i=1}^d m_i = n_x.$
It is clear that the probability of 
\[
\tilde{\lambda}_i = \lambda_i + q_i = \lambda_j + q_j = \tilde{\lambda}_j
\]
holding for any \( i \) and \( j \) is zero.
Consequently, we retain \(d\) distinct eigenvalues with probability one.

According to Theorem \ref{theorem: matrix poly}, \(V(q)\) consists of the coefficients in \eqref{equ: matrix poly} for each \(P_i\). For the \(P_i\) power of $\tilde A(q)$ and an eigenvalue \(\tilde \lambda_j\), the equation \eqref{equJ: matrix poly coef} for the \(i\)-th row of \(V(q)\) is:  
\begin{equation}
\begin{aligned}
&\quad \quad \tilde{\lambda}_j^{P_i}= V_{i1} + V_{i2} \tilde{\lambda}_j + \cdots + V_{in_x} \tilde{\lambda}_j^{n_x-1}, \\
&P_i \tilde{\lambda}_j^{P_i-1} = 0 + V_{i2} + \cdots + V_{in_x} (n_x-1) \tilde{\lambda}_j^{n_x-2}, \\
&\quad \quad \quad \quad \vdots \\
&\prod_{k=0}^{m_j-2}(P_i-k)\tilde{\lambda}_j^{P_i-m_j+1}=\\
&\hspace{2cm}0 + \cdots + V_{in_x} \prod_{k=1}^{m_j-1}(n_x-k) \tilde{\lambda}_j^{n_x-m_j}.
\end{aligned}
\end{equation}

Combining all equations for different $P_i$ and $\tilde \lambda_j$, there are in total \(n_x^2\) unknown variables in $V(q)$ and \(n_x^2\) equations. We now write all equations into a matrix equality.

For the \(P_i\)-th power and the \(j\)-th eigenvalue, to simplify the notation, we define $\Lambda_{ij}$ as
\begin{equation}
\left[\tilde{\lambda}_j^{P_i},\, P_i \tilde{\lambda}_j^{P_i-1},\, \cdots, \, \prod_{k=0}^{m_j-2}(P_i-k) \tilde{\lambda}_j^{P_i-m_j+1}\right],
\end{equation}
 $T_j$ as 
\begin{equation}
 \left[\begin{array}{cccc}
1 & 0 & \cdots & 0 \\
\tilde{\lambda}_j & 1 & \cdots & 0 \\
\tilde{\lambda}_j^2 & 2 \tilde{\lambda}_j & \cdots & 0 \\
\vdots & \vdots & \ddots & \vdots \\
\tilde{\lambda}_j^{n_x-1} & (n_x-1) \tilde{\lambda}_j^{n_x-2} & \cdots & \prod_{k=1}^{m_j-1}(n_x-k) \tilde{\lambda}_j^{n_x-m_j}
\end{array}\right],
\end{equation}
and
\begin{equation}
V_i = \left[V_{i1}, V_{i2}, \cdots, V_{in_x}\right].   
\end{equation}
Thus, we have
$
\Lambda_{ij} = V_i T_j.
$
Stacking \(\Lambda\) and \(T\) into matrices, we have
\begin{equation}
\Lambda = \left[\begin{array}{cccc}
\Lambda_{11} & \Lambda_{12} & \cdots & \Lambda_{1d} \\
\vdots & \vdots & \ddots & \vdots \\
\Lambda_{n_x1} & \Lambda_{n_x2} & \cdots & \Lambda_{n_xd}
\end{array}\right]
\end{equation}
and 
\begin{equation}
T = \left[T_1, T_2, \cdots, T_d\right].
\end{equation}
According to Theorem \ref{theorem: matrix poly}, the matrix \(V(q)\) must satisfy
$$
\Lambda = VT,
$$
where \(\Lambda\), \(V\), and \(T\) are \(n_x\)-by-\(n_x\) matrices and functions of \(q\). The determinant of \(\Lambda\) is the product of the determinants of \(V\) and \(T\). Hence, for \(V\) to be singular, \(\Lambda\) must also be singular.

We establish the non-singularity of \(\Lambda\) through induction. To begin, we construct a sequence of matrices \(C_k \in \MR^{M_k \times M_k}\), for \(k = 1, 2, \cdots, d\) and \(M_k = \sum_{i = 1}^k m_i\). Each \(C_k\) is designed to be a square matrix formed from the top-left corner of the matrix \(\Lambda\). Specifically \(C_k\) is defined as:
\begin{equation}
C_k = \left[\begin{array}{ccc}
\Lambda_{11} & \cdots & \Lambda_{1k} \\
& \vdots & \\
\Lambda_{M_k1} & \cdots & \Lambda_{M_kk}
\end{array}\right]_{M_k \times M_k}.
\end{equation}
As \(n_x = \sum_{i = 1}^d m_i\), where \(d\) is the number of distinct eigenvalues that $A$ has, it follows that \(C_d = \Lambda\).

We initiate induction by proving that the probability of \(C_1\) being singular is zero. Substituting the expression of \(\Lambda_{ij}\) into the definition of \(C_1\), we obtain:

\begin{equation}
\begin{aligned}
&C_1 =\left[\Lambda_{11}^{\top}, \Lambda_{21}^{\top}, \cdots, \Lambda_{M_1 1}^{\top}\right]^{\top} =\\
& \left[\begin{aligned}
&\tilde{\lambda}_1^{P_1} & P_1 \tilde{\lambda}_1^{P_1-1} &\cdots &\prod_{r=0}^{m_1-2}(P_1-r) \tilde{\lambda}_1^{P_1-m_1+1} \\
& & &\vdots \\
&\tilde{\lambda}_1^{P_{m_1}} &P_{m_1} \tilde{\lambda}_1^{P_{m_1}-1} &\cdots &\prod_{r=0}^{m_1-2}(P_{m_1}-r) \tilde{\lambda}_1^{P_{m_1}-m_1+1}
\end{aligned}\right].
\end{aligned}
\end{equation}

The determinant of \(C_1\) is computed as a sum of signed products of matrix entries. Each summand in this sum represents a product of distinct entries, with no two entries sharing the same row or column. Consequently, all summands contribute to the same power order of \(\tilde{\lambda}_1\). Therefore, the determinant of \(C_1\) can be expressed as:
$$
\operatorname{det}(C_1) = c_1 \tilde{\lambda}_1^{\sum_{r=1}^{m_1} P_r - \sum_{r=1}^{m_1}(r-1)},
$$
with constant \(c_1\).

The expression for \(c_1\) can be found by setting \(\tilde{\lambda}_1 = 1\), as this expression of \(C_1\) holds for all realizations of \(\tilde{\lambda}_1\). % Thus:
% $$
% c_1 = \operatorname{det}\left( 
% \left[\begin{aligned}
% &1 &P_1 &\cdots &\prod_{r=0}^{m_1-2}(P_1-r) \\
% & & &\vdots \\
% &1 &P_{m_1} &\cdots &\prod_{r=0}^{m_1-2}(P_{m_1}-r)
% \end{aligned}\right]\right).
% $$
We denote \(E_1\) as the matrix that generates \(c_1\), that is, \(c_1 = \operatorname{det}(E_1)\). Each element in \(E_1\) is a polynomial in \(P_i\), where \(i \in \{1, 2, \cdots, {m_1}\}\). Notably, the polynomials in the same column share the same coefficients. This structure allows \(E_1\) to be transformed into a Vandermonde matrix through column transformations. Given that \(P_i\) are distinct, this Vandermonde matrix is nonsingular, implying \(c_1 \neq 0\). Thus \(C_1\) is nonsingular as long as \(\tilde{\lambda}_1 \neq 0\). Given that \(\tilde{\lambda}_1\) follows a uniform distribution, the probability of \(\tilde{\lambda}_1 = 0\) is zero. Therefore, \(C_1\) is nonsingular with probability one.

Now, assuming that \(C_k\) is nonsingular with probability one, we show that \(C_{k+1}\) is also nonsingular with probability one.
We start by showing that with fixed values of \(\tilde{\lambda}_1, \tilde{\lambda}_2, \cdots, \tilde{\lambda}_k\) and the condition that \(C_k\) being nonsingular, the conditional probability of \(C_{k+1}\) being nonsingular is one. Partition \(C_{k+1}\) into the following blocks:
\begin{equation}
\left[\begin{array}{ll}
C_k & D_{k+1} \\
R_{k+1} & S_{k+1}
\end{array}\right].
\end{equation}
Here, \(R_{k+1} \in \MR^{m_{k+1} \times M_k}\) does not contain \(\tilde{\lambda}_{k+1}\), and \(D_{k+1} \in \MR^{m_{k+1} \times M_k}\) contains terms involving \(\tilde{\lambda}_{k+1}\) up to order \(P_{M_k}\). The terms of \(\tilde{\lambda}_{k+1}\) with power higher than \(P_{M_k}\) are all included in \(S_{k+1} \in \MR^{m_{k+1} \times m_{k+1}}\).

Given the fixed values of \(\tilde{\lambda}_1, \tilde{\lambda}_2, \cdots, \tilde{\lambda}_k\), the determinant \(\operatorname{det}(C_{k+1})\) becomes a polynomial in \(\tilde{\lambda}_{k+1}\) with fixed coefficients. The term with the highest power in this polynomial arises from the product of elements in \(C_k\) with elements in \(S_{k+1}\). The absolute value of the highest power term is in fact 
$
|\operatorname{det}(C_k) \operatorname{det}(S_{k+1})|.
$

 Note that the matrix \(S_{k+1}\) shares the same pattern as \(C_1\). Following a similar line of reasoning as in the case of \(E_1\), we can conclude that the coefficient of \(\operatorname{det}(S_{k+1})\) is also nonzero. Since \(C_k\) is nonsingular, the determinant \(\operatorname{det}(C_{k+1})\) has a nonzero leading term, implying that it has a finite number of roots. Meanwhile, the distribution of \(\tilde{\lambda}_{k+1}\) is a uniform distribution independent of previous eigenvalues. Therefore, the conditional probability of \(\tilde{\lambda}_{k+1}\) coinciding with these roots i.e.
\begin{equation}
\operatorname{Prob}(C_{k+1} \text{ singular}|C_{k} \text{ nonsingular},(\tilde{\lambda}_1, \cdots, \tilde{\lambda}_k) \text{ fixed} ) 
\end{equation}
is zero. Note that \(\operatorname{Prob}(\tilde{\lambda}_1, \tilde{\lambda}_2, \cdots, \tilde{\lambda}_k |C_{k} \text{ nonsingular} )\) is a uniform distribution over the domain of $\tilde{\lambda}_1, \tilde{\lambda}_2, \cdots, \tilde{\lambda}_k,$ which is the set \(\epsilon \mathbb{D}\) projected to the first $k$ dimensional minus a measure zero set representing perturbations that make \(C_k\) singular. By integrating into \(\tilde{\lambda}_1, \tilde{\lambda}_2, \cdots, \tilde{\lambda}_k\), we deduce that:
$$
\operatorname{Prob}(C_{k+1} \text{ singular}|C_{k} \text{ nonsingular}) = 0.
$$
Thus $\operatorname{Prob}(C_{k+1} \text{ singular}) =0,$ as $c_k$ is nonsigular with probability one.

Using induction, \(\Lambda = C_d\) is nonsingular with probability one, which implies that \(V\) is also nonsingular. Consequently, \(S(q)\) has full rank with probability one. Due to the discussion at the start of this proof, the probability of \ac{LCvx} being invalid at all grid points \((N-P_1, N-P_2, \ldots, N-P_{n_x})\) is zero. This means that, for any selection of \(n_x\) grid points, the probability of simultaneously violating the original nonconvex constraints at these points is zero. Hence, the probability of at least \(n_x\) violations of the original nonconvex constraints is also zero.\qed \end{proof}
%%%%%%%%%%%%%%%%%%%%%%%%%%%%%%
\subsection{Proof for Theorem \ref{thm: perturbation error control}}
\label{sec: Proof for perturbation of error control}
\begin{proof}
Since \(\tilde{A}^{i}(\cdot)\) is a smooth function of the perturbation and the matrix norm is a Lipschitz function, there exists an \(\epsilon_1\) such that \(\|\tilde A^{i}(\cdot)\|\) is a Lipschitz function for perturbations \(\|q\| \leq\epsilon_1\) and for \(i = 0, 1, 2, \ldots, N-1\). Let \(L_A\) denote the maximum Lipschitz constant among all \(i\) of \(\|\tilde A^{i}(\cdot)\|\). Additionally, Assumption \ref{assumption: uicompactness} ensures that \(u_i(q)\) is bounded, i.e., \(\max \|u_i(q)\| < D\) for some \(D > 0\). Therefore, when \(\|q\| \leq \epsilon_1\), we can compare \(\tilde{x}_{N+1}(q)\) with \(x_{N+1}(q)\), and obtain:\begin{align*}
&\|\tilde{x}_{N+1}(q) - x_{N+1}(q)\|\\
& = \| \sum_{i = 1}^N (\tilde{A}^{i-1}(0) - \tilde{A}^{i-1}(q)) B u_i(q)\| \\
&\leq \sum_{i = 1}^N L_A \|q\| \|\|B\| \|u_i(q)\| \leq NDL_A\|B\| \|q\|.
\end{align*}
Hence, \(\tilde{x}_{N+1}(q)\) locates in a local neighborhood of \(x_{N+1}(q)\) with radius \(NDL_A \|B\|\epsilon_1\). Meanwhile, since $G$ is an affine function and \(G(x_{N+1}(0)) = 0\), we have \(\|G(\tilde{x}_{N+1}(q))\| < \delta\), for suffciently small $\|q\|$.

Now we move on to the second result. Since the control is bounded, \(x(q)\) is bounded. Thus, \(m(\cdot)\) in Problem~\hyperref[equ:perturbed_zero_order_prototype]{8} is a Lipschitz function within the domain of \(x\) because $m$ is $C^1$ smooth. with Lipschitz constant \(L_m\). Thus:
$$
|m(\tilde{x}_{N+1}(q)) - m(x_{N+1}(q))|\leq NDL_AL_m\|B\| \|q\|.
$$ 

Meanwhile, Lemma \ref{lemma: cQ}, Lemma \ref{lemma: perturbation_normality_stability} (for the compactness), and Theorem \ref{thm: vt continuity} indicate that the optimum value of Problem~\hyperref[equ:perturbed_zero_order_prototype]{8} is continuous with respect to \(q\).  Therefore, for any $\delta>0$ there exists $\epsilon>0$ such that for all $\|q\|<\epsilon,$
\begin{align*}
&m(\tilde{x}_{N+1}(q)) + \sum_{i = 1}^N l_i(g(u_i(q)))\\ &\leq m(x_{N+1}(q)) + \sum_{i = 1}^N l_i(g(u_i(q))) + K_q\|q\| \\
&\leq m(x_{N+1}(0)) + \sum_{i = 1}^N l_i(g(u_i(0))) +K_q\|q\| + \delta/2 \\
&\leq m^*_o + K_q\|q\| + \delta/2,
\end{align*}
where $K_q = NDL_AL_m\|B\| .$
Here, the second inequality arises from Theorem \ref{thm: vt continuity} that the optimum value of Problem~\hyperref[equ:perturbed_zero_order_prototype]{8} is continuous with respect to \(q\) , and that \((u(0), x_{N+1}(0))\) is the solution to the unperturbed problem. The last inequality is derived from the fact that the optimum value of Problem~\hyperref[equ:zero_order_prototype]{4} lower bounds that of Problem~\hyperref[equ:original_LCvx]{3}. Hence, for sufficiently small \(\|q\|\), we have \(m(\tilde{x}_{N+1}(q)) + \sum_{i = 1}^N l_i(g(u_i(q))) \leq m^*_o + \delta.\) 
\qed \end{proof}

\subsection{Proof for Theorem \ref{thm:violation_fixed}}
\label{sec: proof for violation_fixed}
\begin{proof}
 Expanding the dynamics recursion yields
\[
\|\hat{x}_{N+1}-x_{N+1}\|\leq\sum_{i=1}^{N}\|A^{N-i}B(\hat{u}_i-u_i)\|.
\]

Using the relationship between $A$ and $A_c$ with the bound \(\|exp(A_ct)\|\le exp({\|A_c\|t})\), we have
\begin{align}
&\|A^{N-i}B(\hat{u}_i-u_i)\|\\
&\leq e^{\|A_c\| t_f}\|B_c\|\|\hat{u}_i-u_i\|\int_0^{t_f / N}e^{\|A_c\|\tau}d\tau \notag \\
&\leq e^{\|A_c\|t_f}\frac{\|B_c\|}{\|A_c\|}\rho_{\min}(e^{\|A_c\|t_f/N}-1).
\end{align}

Summing over at most \(n_x-1\) violations gives the first bound.

For asymptotically stable \(A_c\), using the additional bound \(\|exp({A_ct})\|\le\beta(A_c) exp({-\alpha t})\) for some $\alpha$ and $\beta$ depending on $A_c$\cite[Fact 15.16.9]{bernstein2018scalar}, we obtain
\[
\|A^{N-i}B(\hat{u}_i-u_i)\|\le e^{\|A_c\|t_f}\|B_c\|\rho_{\min}\beta(A_c)\frac{t_f}{N},
\]
leading directly to the second result.   
\end{proof}

%%%%%%%%%%%%%%%%%%
\subsection{Proof for Theorem \ref{thm: over_relaxed main}}
\label{sec: proof for over_relax}
\begin{proof}

Denote $(x(q), u(q), \sigma(q))$ the solution toProblem~\hyperref[equ:LCvxallmissed]{9}  generated by $t_s^*$ and perturbed by $q$. For any $\epsilon > 0$, we set $\eta = \min\{\epsilon  L_l(t_f - t_s^*)/N, \frac{3\epsilon}{4}\}$. According to Theorem \ref{thm: bisection search method}, there exists a $t_s^*$ such that 
\begin{equation}
\begin{aligned}
v\left(t_s^*\right)= & m\left(x_{N+1}(0)\right)+\frac{t_f-t_s^*}{N} \sum_{i=1}^N l\left(\sigma_i(0)\right) \\
& +t_s^* l\left(\rho_{\min }\right) \leq m^*+t_f l\left(\rho_{\min }\right)+\frac{\eta}{2}
\end{aligned}
\end{equation}
where $m^*$ is the optimum of Problem~\hyperref[equ: boundaryonly]{7} and $x_{N+1}(0)$ is the result of the unperturbedProblem~\hyperref[equ:LCvxallmissed]{9}  formed by $t_s^*$.

The first result follows from Theorem \ref{thm: perturbation error control}. The second result is from Theorem \ref{theorem: main_theorem} as long as we choose $\delta$ sufficiently small such that $\|q\|\leq \delta$ is included in the cube area defined by Theorem \ref{theorem: main_theorem}. We now proceed to prove the third and fourth results.

According to Lemma \ref{lemma: perturbation_normality_stability}, 
\begin{equation}
m\left(x_{N+1}(q)\right) + \frac{t_f-t_s^*}{N} \sum_{i=1}^N l\left(\sigma_i(q)\right) + t_s^*  l\left(\rho_{\text{min}}\right)
\end{equation}
is a continuous function of $q$ in a neighborhood of zero. Therefore, there exists a $\delta > 0$ such that for all $\|q\| \leq \delta$,
\begin{equation}
\begin{aligned}
&m\left(x_{N+1}(q)\right)+\frac{t_f-t_s^*}{N} \sum_{i=1}^N l\left(\sigma_i(q)\right)\\
&\leq v\left(t_s^*\right)-t_s^* l\left(\rho_{\min }\right)+\frac{\eta}{2} \\
& \leq m^*+\left(t_f-t_s^*\right) l\left(\rho_{\min }\right)+\eta .
\end{aligned}
\end{equation}

Since $x_{N+1}(q)$ is feasible for the perturbated problem, we have $G(x_{N+1}(q)) = 0$ and $x_{N+1}(q)$ is feasible for Problem~\hyperref[equ: boundaryonly]{7}. Therefore, $m( x_{N+1}(q)) \geq m^*$ and 
\begin{equation}
\frac{\sum_{i=1}^N l\left( \sigma_{i}(q)\right)}{N} \leq \frac{\eta}{t_f-t_s^*} + l\left(\rho_{\text{min}}\right).
\end{equation} As $\sigma_i(q) \geq \rho_{\text{min}}$ for all $i$ and $l$ has positive derivative, we have
\begin{equation}
l\left(\rho_{\text{min}}\right) \leq l\left(\sigma_i(q)\right) \leq l\left(\rho_{\text{min}}\right) + \frac{N \eta}{t_f - t_s^*}.
\end{equation}
Given that $\eta \leq \epsilon  L_l(t_f - t_s^*)/N$ and the derivative of $l$ is larger than $L_l$, we have
\begin{equation}
\left|\sigma(q) - \rho_{\text{min}}\right| \leq \frac{N\eta}{L_l(t_f - t_s^*)} \leq \epsilon,
\end{equation}
yielding the third result.

For the final result, since $l\left(\rho_{\text{min}}\right) \leq l\left(\sigma_i(q)\right)$, we have

$m\left(x_{N+1}(q)\right) \leq m^* + \eta.$
Furthermore, Theorem \ref{thm: perturbation error control} indicates that
\[
|m(\tilde x_{N+1}(q)) - m(x_{N+1}(q))| \leq N D L_A L_M \|B\|\|q\|,
\]
where $D$, $L_A$, and $L_M$ are defined in Theorem \ref{thm: perturbation error control} and are constants when $t_s^*$ is fixed. Thus, we obtain
\begin{equation}
m\left(\tilde{x}_{N+1}(q)\right) \leq m^* + \eta + N D L_A L_M \|B\| \|q\|.
\end{equation}
With a sufficiently small $\|q\|$, we have
\(
m\left(\tilde{x}_{N+1}(q)\right) \leq m^* + \epsilon.
\)
\qed \end{proof}
%%%%%%%%%%%%%%%%%%%%%%%%%%%%%
%%%%%%%%%%%%%%%%%%%%%%%%%%%%%
% \red \section{Acknowledge}
\bibliographystyle{plainnat} % We choose the "plain" reference style
\bibliography{refs} % Entries are in the refs.bib file
\end{document}